\documentclass[11pt,reqno]{article}
\usepackage[table]{xcolor}
\usepackage{amsmath}
\usepackage{amsthm}
\usepackage{amssymb}
\usepackage{epsfig}
\usepackage{epstopdf}
\usepackage{calc}
\usepackage{xspace}
\usepackage{comment}
\usepackage{paralist}
\usepackage{verbatim}
\usepackage{url}
\usepackage[normalem]{ulem}
\usepackage{setspace}
\usepackage{rotating}
\usepackage{amsmath,amssymb}
\usepackage{subfigure}
\usepackage{array,arydshln}	
\usepackage{multirow}
\usepackage{booktabs}
\usepackage{color,colortbl}
\usepackage{authblk}
\usepackage{hyperref}
\usepackage{enumitem}
\usepackage{tabularx}	
\usepackage{adjustbox}
\usepackage{float}
\usepackage{makecell}
\usepackage{bm}
\usepackage{mathrsfs}
\usepackage{indentfirst}
\usepackage{caption}
\graphicspath{ {./pictures of paper/} }

\usepackage{natbib}
\DeclareMathOperator*{\argmax}{argmax}

\newtheorem{theorem}{Theorem}
\newtheorem{prop}{Proposition}

\newtheorem{corollary}{Corollary}
\newtheorem{lemma}{Lemma}
\newtheorem{@remark}{\bf Remark}
\newenvironment{remark}{\begin{@remark}\itshape}{\end{@remark}}

\newcommand {\A} {\alpha}
\newcommand {\B} {\beta}
\newcommand {\W} {\omega}
\newcommand {\G} {\gamma}
\newcommand {\E} {\mathrm{E}}

\newcommand {\T} {\theta}

\newcommand {\ep} {\epsilon}

\newcommand{\tcr}{\textcolor{red}}

\newcommand{\tsig}{\tilde{\sigma}}

\newcommand{\pa}{\partial}
\newcommand {\paa}{\partial^2_{\T\T'}}

\newcommand{\mF}{\mathcal{F}}
\newcommand{\bX}{\mathbf{X}_n}

\title{ Robust Bayesian estimation in conditionally heteroscedastic time series
models}
\author{Jeongho Lee}
\affil{Department of Statistics, Kyungpook National University}
\author[1]{Junmo Song}

\begin{document}

\maketitle

\begin{abstract}
Outliers can seriously distort statistical inference by inducing excessive sensitivity in the likelihood function, thereby compromising the reliability of Bayesian estimation. To address this issue, we develop a robust Bayesian estimation method for conditionally heteroscedastic time series models by extending the density power divergence (DPD) framework to the Bayesian setting. The resulting DPD-based posterior distribution, controlled by a tuning parameter, achieves a smooth balance between efficiency and robustness. We establish the asymptotic properties of the proposed estimator; specifically, the DPD-based posterior is shown to satisfy a Bernstein–von Mises type theorem, converging to a normal distribution centered at the minimum DPD estimator (MDPDE). Furthermore, the corresponding Bayes estimator, defined as the posterior mean under the DPD-based posterior (EDPE), is asymptotically equivalent to the MDPDE. Monte Carlo simulations based on GARCH(1,1) models confirm that the proposed EDPE performs well under both uncontaminated and contaminated data, maintaining robustness where the ordinary Bayes estimator becomes severely biased. An empirical application to BTC-USD returns further demonstrates the practical advantages of the proposed robust Bayesian framework for financial time series analysis.
\end{abstract}

\section{Introduction}

The presence of outliers can significantly affect the accuracy and reliability of statistical inference. 
Outliers arise in a wide range of applications, including finance, economics, and environmental studies. 
To mitigate their impact, robust statistical methods have been extensively studied within the frequentist framework (see, for example, \cite{huber1981robust, ronchetti2021robust, loh2024robust}). A well-established finding from this line of research is that likelihood-based inference is particularly sensitive to outlying observations. Because the likelihood function plays a central role in Bayesian updating, this sensitivity naturally carries over to Bayesian inference. Consequently, despite the widespread use of Bayesian inference due to its favorable properties, the Bayes estimator under quadratic loss, when based on the ordinary posterior distribution, remains highly vulnerable to data contamination.

In Bayesian analysis, a traditional strategy for achieving robustness is to modify the assumed data-generating mechanism, typically by adopting heavy-tailed distributions for the model errors (see, for example, \cite{andrade2011bayesian, ohagan2012conflict}). While such approaches can mitigate the influence of extreme observations, they may lead to non-negligible efficiency loss when contamination is mild or absent. Motivated by this trade-off, a growing literature has sought to import robust frequentist ideas into the Bayesian framework through the construction of pseudo-posterior distributions. Notable examples include robust Bayesian procedures based on the Hellinger distance \cite{hooker2014bayesian}, the density power divergence \cite{ghosh2016robust}, and the $\gamma$-divergence \cite{nakagawa2020robust}. Related approaches have also been developed using robust estimating equations within likelihood-free or approximate Bayesian frameworks, such as the ABC-based robust posterior proposed by \cite{ruli2020robust}. Together, these methods yield Bayesian estimators that incur only mild efficiency loss under the model while exhibiting stability in the presence of outliers.

Despite this progress and the widespread use of time series models in various fields, most existing robust Bayesian methods have been developed for independent and identically distributed data or for regression-type models. In contrast, robust Bayesian inference for time series data remains comparatively underexplored, even though temporal dependence and dynamic model structures can substantially amplify the impact of outliers. 
However, extending the aforementioned Bayesian methods to time series settings is not straightforward. This is mainly because it requires careful treatment of conditional likelihoods, dependence structures, and recursive components inherent in many time series models.

In this study, we address this gap by developing a divergence-based robust Bayesian framework for conditionally heteroscedastic time series models. Specifically, we extend the density power divergence (DPD) approach of \cite{ghosh2016robust} to a broad class of models characterized by time-varying conditional variances. Originally introduced by \citet{basu1998}, the DPD provides a smooth interpolation between the Kullback--Leibler divergence and the $L_2$ distance, thereby offering a principled mechanism to balance efficiency and robustness. 
This divergence-based approach has been successfully used to develop robust estimators in the frequentist context (see, for example, \cite{fujisawa2008robust, lee2009minimum, ghosh2013robust, song2017robusta}). 
As also shown in \citet{ghosh2016robust}, the pseudo-posterior induced by the DPD inherits these desirable properties in i.i.d. settings. Building on this insight, we adapt the DPD-based Bayesian methodology to time series data, establish its large-sample properties under stationarity and ergodicity, and demonstrate its effectiveness through simulation studies and an empirical analysis of financial time series data.

The main contributions of this paper are threefold. First, we formulate a DPD-based pseudo-posterior for conditionally heteroscedastic time series models and define the associated Bayes estimator as the posterior mean. Second, under suitable regularity conditions, we establish the consistency and asymptotic normality of both the corresponding frequentist estimator and the DPD-based posterior distribution, including a Bernstein--von Mises type result. Third, through Monte Carlo experiments and an application to Bitcoin return data, we show that the proposed approach achieves substantial robustness gains with mild efficiency loss in uncontaminated settings.

The remainder of the paper is organized as follows. Section \ref{Main results} introduces the proposed DPD-based Bayesian estimation procedure and presents the main theoretical results. Section \ref{Simulation study} reports the results of a Monte Carlo study assessing finite-sample performance. Section \ref{Real data analysis} applies the method to Bitcoin return data, and Section \ref{Conclusion} concludes.

\section{Main results}\label{Main results}

Consider the following time series model with parameter $\theta \in \mathbb{R}^d$:
\begin{equation}\label{gts} 
X_t = \sigma_t(\theta)\,\epsilon_t,
\end{equation}
where $\sigma_t^2(\theta)$ denotes the conditional variance given $\mF_{t-1} = \sigma(X_s \mid s \le t-1)$, and 
$\{\epsilon_t \mid t \in \mathbb{Z}\}$ is a sequence of i.i.d.\ random variables with density $f_\epsilon$,  having zero mean and unit variance. 
 We assume that $\{X_t \mid t \in \mathbb{Z}\}$ generated by the above model with true parameter $\T_0$ is strictly stationary and ergodic. A broad class of stationary time series models falls within this framework, encompassing standard GARCH models as well as nonlinear and asymmetric extensions such as power-transformed and threshold GARCH models.

When estimating the model based on the observations $\{X_1,\cdots,X_n\}$, the conditional variance $\sigma_t^2(\theta)$ is often not explicitly obtainable due to the issue of initial values. Since $\sigma_t^2(\theta)$ is typically defined through a recurrence equation, it can be approximated by recursion with appropriately chosen initial values. See, for example, \cite{francq:zakoian:2004} for standard GARCH models and \cite{hamadeh2011asymptotic} for nonlinear GARCH models. In this study, we assume that a proxy for $\sigma_t^2(\theta)$ is obtained using this recursive method. Hereafter, we denote the resulting approximated process by $\{\tilde{\sigma}_t^2(\theta) \mid t = 1,\cdots,n\}$.

The feasible likelihood constructed using the proxy of the conditional variance is given by
\[
\tilde L(\theta \mid \bX) = \prod_{t=1}^n \tilde f_\theta(X_t \mid \mF_{t-1}),
\]
where $\bX = (X_1,\cdots,X_n)$ and $\tilde f_\T(x | \mF_{t-1})=\frac{1}{\tilde\sigma_t(\theta)}f_\ep\big(\frac{x}{\tilde \sigma_t(\theta)}\big)$. The posterior distribution corresponding to model~\eqref{gts} is then given by
\begin{equation}\label{posterior}
\pi(\theta \mid \bX) \propto \tilde L(\theta \mid \bX)\,\pi(\theta),
\end{equation}
where $\pi(\theta)$ denotes a prior distribution. Under squared error loss, the corresponding Bayes estimator, referred to as the expected ordinary posterior estimator (EOPE), is given by
\begin{equation}\label{EOPE}
   \hat{\theta}_n^{EOPE} = \int \theta\, \pi(\theta \mid \bX)\, d\theta.
 \end{equation}
As noted in the introduction, outliers can severely distort the likelihood, leading to an unreliable posterior distribution and, consequently, a biased Bayes estimator. This effect is illustrated in the simulation study below. To address this issue, we propose a robust Bayesian estimator using a divergence based approach.

\subsection{Density power divergence based Bayesian estimation}
In this subsection, we review the robust Bayesian estimation method based on the DPD introduced by \cite{ghosh2016robust} and extend it to the time series model (\ref{gts}).

The DPD, proposed by \cite{basu1998}, is a divergence measure that quantifies the discrepancy between two density functions. Specifically, for density functions $f$ and $g$, the DPD between them is defined as follows:
\begin{eqnarray*}\label{DPD}
d_\G (g, f)=\left\{\begin{array}{lc}
\displaystyle\int\Big\{f^{1+\G}(x)-\Big(1+\frac{1}{\G}\Big)\,g(x)\,f^\G(x)+\frac{1}{\G}\,g^{1+\G}(x)\Big\} dx &,\G>0, \\[20pt]
\displaystyle\int g(x)\big\{ \log g(x)-\log f(x) \big\} dx
&,\G=0.
\end{array} \right.
\end{eqnarray*}
Consider a family of parametric densities $\{ f_{\theta} \mid \theta \in \Theta \}$, and let $X_1, \cdots, X_n$ be a random sample with true density $f_{\T_0}$. The minimum DPD estimator (MDPDE) is defined as the minimizer of the empirical version of the divergence $d_\G(f_{\T_0},f_\T)$. Specifically,
\begin{eqnarray}\label{MDPDE0}
\hat \theta_{\G, n} = \argmax_{\theta \in \Theta}\, Q_n^{(\G)}(\T),
\end{eqnarray}
where \begin{eqnarray*}
Q_n^{(\G)}(\T) = \left\{ \begin{array}{ll}
   \displaystyle    \frac{1}{\G} \sum_{t=1}^n
     f_\theta^{\G}(X_t)- \frac{n}{1+\G}\int f_\theta^{1+\G}(x) dx    & \mbox{, $\G > 0$,}\\[20pt]
   \displaystyle  \sum_{t=1}^n\log f_\theta(X_t)      & \mbox{, $\G = 0$.}
   \end{array}
 \right.
\end{eqnarray*}
The tuning parameter $\G$ controls the trade-off between efficiency and robustness. In particular, the MDPDE with $\G$ close to 0 is known to retain high efficiency while providing robustness. We also note that the MDPDE with $\G=0$ is the MLE. 

Based on the above, \cite{ghosh2016robust} introduce the so-called $\G$-robustified posterior density by replacing the log-likelihood
with $Q_n^{(\G)}$ in Bayes' formula. Specifically, their robust posterior is obtained as follows:
\begin{align*}
    \pi_{\G}(\theta\mid\bX)
     & \propto \exp \big(Q_n^{(\G)}(\T)  \big) \pi(\theta),
\end{align*}
where $\pi(\T)$ is a prior of $\T$.
Note that the posterior distribution above reduces to the ordinary posterior when $\G=0$. 
\begin{remark}
It can be shown that, for each fixed $\theta\in\Theta$,
\[
Q_n^{(\gamma)}(\theta)=\frac{n}{\gamma}+\sum_{t=1}^n \log f_\theta(X_t)
-n+o(1)\qquad \text{as } \gamma \downarrow 0.
\]
Consequently, we have
\[
Q_n^{(\gamma)}(\theta) - \frac{n}{\gamma}
\longrightarrow
\sum_{t=1}^n \log f_\theta(X_t) - n,
\]
pointwise in $\theta$.
Since posterior densities are defined only up to a normalizing constant, the diverging term $n/\gamma$ and the additive constant $-n$ do not affect the
posterior density up to normalization. The above expansion therefore suggests, at a heuristic level, that the DPD-based posterior smoothly connects to the
ordinary Bayesian posterior as $\gamma \downarrow 0$.
However, the pointwise convergence above is not sufficient to justify that the DPD-based posterior and the ordinary posterior have similar distributional
shapes. Such an interpretation requires stronger notions of convergence, for example uniform convergence of the unnormalized posterior densities or, more
strongly, convergence in total variation. These issues are addressed under suitable conditions in Theorem~\ref{thm:gamma0} below for the time series model
\eqref{gts}.
\end{remark}

This type of pseudo-posterior can also be derived in time series models by extending the DPD for i.i.d. setting to dependent data. This extension is achieved using the following conditional version of the DPD, which compares the two conditional densities $f_{\T_0}(\cdot\mid\mathcal{F}_{t-1})$ and $f_\T(\cdot\mid\mathcal{F}_{t-1})$:
\begin{eqnarray*}
&&d_\G\left(f_{\T_0}(\cdot\mid\mathcal{F}_{t-1}),f_\T(\cdot\mid\mathcal{F}_{t-1})\right)\\
&&\\
&&=\left\{\begin{array}{ll}
                         \displaystyle \int\left\{ f_\theta^{1+\G}(x|\mathcal{F}_{t-1})-\Big(1+\frac{1}{\G}\Big)f_{\theta_0}(x|\mathcal{F}_{t-1})f_\theta^\G(x|\mathcal{F}_{t-1})+\frac{1}{\G}f_{\theta_0}^{1+\G}(x|\mathcal{F}_{t-1}) \right\}dx &, \G>0\\[20pt]
                          \displaystyle \int f_{\theta_0}(x|\mathcal{F}_{t-1})
                           \big\{\log f_{\theta_0}(x|\mathcal{F}_{t-1})-\log
                           f_\theta(x|\mathcal{F}_{t-1})\big\}dx&,\G=0.
                        \end{array}\right.
\end{eqnarray*}
Similarly to how the MDPDE in (\ref{MDPDE0}) is obtained, one can derive the MDPDE for the model (\ref{gts}) as follows:
\begin{eqnarray}\label{MDPDE}
\hat \theta_{\G, n} = \argmax_{\theta \in \Theta}\, \widetilde Q_n^{(\G)}(\T),
\end{eqnarray}
where
\begin{eqnarray}\label{Q}
\widetilde Q_n^{(\G)}(\T) = \left\{ \begin{array}{ll}
   \displaystyle    \frac{1}{\G} \sum_{t=1}^n
     \tilde f_\theta^{\G}(X_t|\mathcal{F}_{t-1})- \frac{1}{1+\G}\sum_{t=1}^{n}\int \tilde f_\theta^{1+\G}(x|\mathcal{F}_{t-1}) dx &, \G > 0,\\[20pt]
   \displaystyle  \sum_{t=1}^n\log \tilde f_\theta(X_t|\mathcal{F}_{t-1})      &, \G = 0 
   \end{array}
 \right.
\end{eqnarray}
(cf. subsection 3.1 in \cite{song2021sequential}).
Using the objective function $\widetilde Q_n^{(\gamma)}(\T)$ defined above, the $\gamma$-robust posterior density for the time series model \eqref{gts} is given
by
\begin{align*}
     \pi_{\G}(\theta\mid\bX)
     & \propto \exp \big(\widetilde Q_n^{(\G)}(\T)  \big) \pi(\theta).
\end{align*}
Correspondingly, the DPD analogue of the EOPE \eqref{EOPE}, referred to as the
expected DPD-based posterior estimator (EDPE), is defined as
\begin{align*}
    & \hat{\theta}_{\G,n}^{EDPE} = \int \theta\,  \pi_{\G}(\theta\mid\bX)d\theta.
\end{align*}
\begin{remark} The integral included in $\widetilde{Q}_n^{(\G)}(\T)$ can pose a computational burden. However, when the support of $f_\epsilon$ is $\mathbb{R}$, we can deal with the integral in advance. Specifically, observe that  
\[
\int \tilde f_\theta^{1+\G}(x|\mathcal{F}_{t-1})dx= \frac{1}{\tilde{\sigma}^\G_t(\theta)} \int f_\epsilon^{1+\G}(x)dx.
\]
Since the integral $\int f_\epsilon^{1+\G}(x) \, dx$ does not involve any parameters, it can be pre-computed either analytically or numerically for each $\G$. For instance, if $\epsilon_t$ follows $N(0,1)$, the integral evaluates to $(2\pi)^{-\gamma/2}\,(1+\G)^{-1/2}$. Using this pre-calculated value in place of the integral reduces the computational burden significantly.
\end{remark}

As mentioned earlier, the tuning parameter $\gamma$ controls the robustness of the MDPDE against outliers. In the same spirit, $\gamma$ also governs the
robustness of the DPD-based posterior, and consequently the EDPE above. This characteristic stems from the connection between the DPD-based posterior and the ordinary Bayesian posterior. 
To establish this connection more precisely, we impose the following regularity conditions.
\begin{enumerate}
\item[\bf C1.] The innovation density $f_\ep$ is strictly positive on $\mathbb{R}$.
\item[\bf C2.]
For each $x$ and $t$, $\tilde f_\theta(x \mid \mathcal F_{t-1})$ is continuous in $\theta$.
\item[\bf C3.]
For all sufficiently small $\gamma>0$,
\[
\sup_{\theta \in \Theta}
\left|
\widetilde Q_n^{(\gamma)}(\theta)-
\sum_{t=1}^n \log \tilde f_\theta(X_t \mid \mathcal F_{t-1})
\right|
\le V_n \gamma,
\]
where $V_n>0$ is a random variable that does not involve $\theta$.
\item[\bf C4.]
The prior density $\pi(\theta)$ is continuous on a compact parameter space $\Theta$.
\item[\bf C5.]
 For some $\theta^*\in\Theta$, $\pi(\theta^*)>0$.
\end{enumerate}
Condition {\bf C1} is standard in time series models such as GARCH type models, where the innovation density is typically taken to have full support on $\mathbb{R}$ (e.g., Gaussian or Student-$t$ distributions). It ensures that $\tilde f_\theta(x\mid\mathcal F_{t-1})$ and $\tilde L(\theta\mid\bX)$ are strictly positive.
Condition {\bf C5} is required to ensure the positivity of the normalizing constant of the ordinary posterior $\pi(\theta \mid \bX)$. When combined with
the positivity of the likelihood $\tilde L(\theta \mid \bX)$ on $\Theta$, this condition guarantees that the normalizing constant of the
limiting posterior is strictly positive.

\begin{theorem}\label{thm:gamma0}
Suppose that conditions {\bf C1} -- {\bf C4} hold and the parameter space $\Theta$ is compact.
Then, the unnormalized posterior densities satisfy the following uniform convergence
\[
\sup_{\theta \in \Theta}
\left|
\exp\!\big(\widetilde Q_n^{(\gamma)}(\theta)\big)\pi(\theta)
-
\tilde L(\theta \mid \bX)\pi(\theta)
\right|
\longrightarrow 0\quad as\ \gamma\downarrow 0.
\]
Additionally, if condition {\bf C5} holds, the DPD-based posterior 
$\pi_\gamma(\theta \mid \bX)$ converges to the ordinary posterior
$\pi(\theta \mid \bX)$ in total variation, that is,
\[
\frac12\int_\Theta |\pi_\gamma(\theta\mid\mathbf X_n)-\pi(\theta\mid\mathbf X_n)|\,d\theta
\longrightarrow 0\quad as\ \gamma\downarrow 0.
\]
\end{theorem}
These limiting results imply that the DPD-based posterior constitutes a continuous robustification of the ordinary Bayesian posterior. When $\gamma$ is  very close to zero, $\pi_\gamma(\theta\mid \bX)$ remains highly similar to the ordinary posterior $\pi(\theta\mid \bX)$, thereby retaining high efficiency. At the same time, any $\gamma>0$ downweights the influence of extreme observations, providing robustness analogous to that of the MDPDE. Hence the tuning parameter $\gamma$ controls a smooth trade-off between efficiency and robustness also in Bayesian inference.

As a direct consequence of Theorem~\ref{thm:gamma0}, the convergence of the posterior distributions in total variation, together with the compactness of $\Theta$, implies the convergence of the corresponding posterior expectations. This leads to the following corollary.
\begin{corollary}\label{cor:edpe}
Under the conditions of Theorem~\ref{thm:gamma0}, we have
\[
\hat\theta^{EDPE}_{\gamma,n}
=\int_\Theta \theta\,\pi_\gamma(\theta\mid\bX)\,d\theta\
\longrightarrow\
\hat\theta^{EOPE}_{n}=\int_\Theta \theta\,\pi(\theta\mid\bX)\,d\theta
\quad as\ \gamma\downarrow 0.
\]
\end{corollary}


\begin{remark}\label{rmk:C3}
Although Condition {\bf C3} in Theorem \ref{thm:gamma0} is stated in terms of the objective function $\widetilde Q_n^{(\gamma)}(\theta)$, in applications it is often convenient to work with a shifted version. Specifically, for any function $K_n(\gamma)$ that is independent of $\theta$, define
\[
\widetilde Q_{n,K}^{(\gamma)}(\theta):=\widetilde Q_n^{(\gamma)}(\theta)-K_n(\gamma).
\]
Then $\widetilde Q_{n,K}^{(\gamma)}(\theta)$ and $\widetilde Q_n^{(\gamma)}(\theta)$ induce the same posterior distribution. Consequently, we have
\[
\pi_\gamma(\theta\mid\bX)\propto \exp\!\big(\widetilde Q_n^{(\gamma)}(\theta)\big)\pi(\theta)
\propto \exp\!\big(\widetilde Q_{n,K}^{(\gamma)}(\theta)\big)\pi(\theta),
\]
so adding or subtracting a $\theta$-independent term does not alter the resulting posterior distribution. Therefore, when verifying Condition~\textbf{C3} for a given model, one may equivalently establish it for $\widetilde Q_{n,K}^{(\gamma)}(\theta)$ rather than for $\widetilde Q_n^{(\gamma)}(\theta)$ itself.
\end{remark}

\begin{remark}\label{rmk:gamma}
The tuning parameter $\gamma$ plays a role that is fundamentally different from that of the model parameters $\T$. While the components of $\T$ are associated with the data-generating dynamics and possess true underlying values, the parameter $\gamma$ does not admit a true value.
Instead, for a given data set, an optimal choice of $\gamma$ can be defined only relative to a specific inferential or predictive objective and reflects a procedure-level trade-off between robustness and efficiency.

From a Bayesian perspective, one may therefore view $\gamma$ as a decision or design parameter rather than a genuine model parameter. Although it is conceptually possible to assign a prior distribution to $\gamma$ and estimate it jointly with $\T$, doing so would change the nature of the problem itself, as the resulting posterior would no longer correspond to inference on parameters of the data-generating process.
Moreover, the Bernstein--von Mises type results established in the next subsection rely crucially on the frequentist asymptotic normality of the MDPDE
with $\gamma$ treated as a fixed tuning constant. For these reasons, and to preserve a clear theoretical interpretation of the asymptotic analysis, we treat $\gamma$ as fixed throughout the theoretical development.
\end{remark}

\begin{remark}\label{rmk:select_gamma}
In empirical applications, selecting an appropriate value of $\gamma$ is an important practical issue. In the frequentist literature on MDPD estimation, $\gamma$ is typically treated as a nuisance robustness parameter, and its choice is guided by the trade-off between efficiency and robustness; see, for example, \cite{warwick2005choosing} and \cite{fujisawa2008robust}. In practice, relatively small values of $\gamma$ are often recommended, since
excessively large values may lead to a substantial loss of efficiency when the degree of contamination is limited. More importantly, such values of $\gamma$ are typically sufficient to achieve a favorable balance between robustness and efficiency; see, for instance, \cite{lee2009minimum} and \cite{song2017robusta}.

From a Bayesian decision-theoretic perspective, an optimal value of $\gamma$ can also be defined for a given data set by minimizing an out-of-sample predictive loss. In particular, selecting $\gamma$ based on the RMSE of one-step-ahead volatility forecasts provides a simple and effective data-driven rule that directly reflects the robustness--efficiency trade-off in terms of predictive performance.In the real data analysis of Section \ref{Real data analysis}, we adopt this approach and select $\gamma$ by minimizing the out-of-sample forecasting RMSE, which yields values of $\gamma$ in the range $[0.2,0.3]$, consistent with empirical findings in the existing literature.
\end{remark}

\subsection{Asymptotic Properties}\label{sub:asymp}
To establish the asymptotic properties of $\hat\T_{\G,n}^{EDPE}$, we impose the regularity conditions. For this, let $f_\T(x|\mathcal{F}_{t-1})$ denote the infeasible conditional density obtained by replacing $\tilde \sigma_t(\T)$ in $\tilde f_\T(x|\mathcal{F}_{t-1})$ with $\sigma_t(\T)$. 
Although the same notation $Q_n^{(\gamma)}(\theta)$ was used in \eqref{MDPDE0} for the i.i.d.\ setting, we now use $Q_n^{(\gamma)}(\theta)$ to denote the counterpart of $\widetilde Q_n^{(\G)}$ in \eqref{MDPDE}, replacing $\tilde f_\T(x|\mathcal{F}_{t-1})$ with $f_\T(x|\mathcal{F}_{t-1})$. Further, let $\tilde q_t^{(\G)}(\T)$ be the function satisfying $\widetilde Q_n^{(\G)}(\T)=\sum_{t=1}^n \tilde q_t^{(\G)}(\T)$, and define $q_t^{(\G)}(\T)$ in the same manner. Hereafter, $\| \cdot\|$ denotes the $l_1$ norm for matrices and vectors. For notational convenience, we use $\pa_{\theta}$ and $\paa$ to denote $\frac{\pa}{\pa\T}$ and $\frac{\pa^2}{\pa\theta \pa \theta'}$, respectively. We now introduce the following conditions:
\begin{enumerate}
\item[\bf A1.]  $q_t^{(\G)}(\T)$ is continuous in $\T$, and $\{q_t^{(\G)}(\T)|t\in\mathbb{Z}\}$ is strictly stationary and ergodic for each $\T\in\Theta$.
\item[\bf A2.]
$\displaystyle\E \sup_{\T\in\Theta} \big|q_t^{(\G)}(\T)\big|<\infty$ and $\displaystyle \frac{1}{n}\sum_{t=1}^n \sup_{\T\in\Theta} \big| q_t^{(\G)}(\T)-\tilde q_t^{(\G)}(\T)\big|=o(1)\quad a.s.$ 
\item[\bf A3.] $\{\pa_\T q_t^{(\G)}(\T_0), \mathcal{F}_t  \mid t \in \mathbb{Z}\}$ is a martingale difference sequence.
\item[\bf A4.]$\displaystyle \frac{1}{\sqrt{n}}\sum_{t=1}^n \big\| \pa_\T\,q_t^{(\G)}(\T_0)-\pa_\T\,\tilde q_t^{(\G)}(\T_0)\big\|= o(1)\quad a.s.$
\item[\bf A5.]For some neighborhood $N_1(\T_0)$ of $\T_0$, $$\displaystyle \frac{1}{n}\sum_{t=1}^{n}\sup_{\theta \in N_1(\T_0)}\big\|\paa  q_t^{(\G)}(\T)-\paa  \tilde q_t^{(\G)}(\T)\big\|= o(1)\quad a.s.$$
\item[\bf A6.] For some neighborhood $N_2(\T_0)$ of $\T_0$,
 $$\displaystyle\E \sup_{\theta \in N_2(\T_0)}\big\|\paa  q_t^{(\G)}(\T)\big\|  <\infty.$$
\item[\bf A7.] The matrices $\mathcal{J}_\G=-\E \big[ \paa  q_t^{(\G)}(\T_0)\big]$ and $\mathcal{I}_\G=\E \big[ \pa_\T q_t^{(\G)}(\T_0) \pa_{\T'} q_t^{(\G)}(\T_0)\big]$ exist and $\mathcal{J}_\G$ is positive definite.
\end{enumerate} 

Assumptions \textbf{A1--A2} ensure a uniform law of large numbers for the objective function and control the discrepancy between the infeasible criterion based on $f_\T(x|\mathcal{F}_{t-1})$ and its feasible counterpart based on $\tilde f_\T(x|\mathcal{F}_{t-1})$.  Assumptions \textbf{A3--A4} provide a martingale difference structure that leads to a central limit theorem. Assumptions \textbf{A5--A6} serve as technical conditions for establishing uniform convergence of the second derivative $\paa \widetilde Q_n^{(\G)}(\T)$.

These assumptions closely follow the standard regularity conditions used to establish consistency and asymptotic normality of M-estimators in dependent time series models; see, for example, \cite{ling2010general}. For DPD based M-estimation in time series settings, similar conditions are employed in \cite{song2021sequential}. Importantly, the same set of assumptions is also used in the proof of the Bernstein--von Mises type results presented below.

We first establish the consistency and asymptotic normality of the MDPDE $\hat\theta_{\gamma,n}$ from a frequentist perspective. We then turn to the Bayesian counterpart and study the asymptotic behavior of the DPD-based posterior distribution.

\begin{theorem}\label{thm1}
Suppose that assumptions {\bf A1} and {\bf A2} hold. Then, $\hat\T_{\G,n}$ converges almost surely to $\T_0$. Furthermore, if assumptions {\bf A3}–{\bf A7} are also satisfied, we obtain
\[\sqrt{n}(\hat\T_{\G,n}-\T_0)\ \stackrel{d}{\longrightarrow}\ N_d(0, \mathcal{J}^{-1}_{\G} \mathcal{I}_{\G}\mathcal{J}^{-1}_{\G}).\]
\end{theorem}

\begin{theorem}\label{thm2}
Suppose that assumptions {\bf A1}, {\bf A2}, and {\bf A7} are satisfied and let $\pi(\theta)$ be any prior density that is
positive and continuous at $\T_0$. Then, it holds that
\begin{equation*}
     \int_{\mathbb{R}^d} \Big| \pi_\G(t|\bX) - \frac{|\mathcal{J}_\G|^{1/2}}{(2\pi)^{d/2}} e^{- \frac{1}{2} t' \mathcal{J}_\G t} \Big| dt\ \stackrel{a.s.}{\longrightarrow}\ 0,
\end{equation*}
where $ \pi_\G(t|\bX)$ is the $\G$-posterior density of $t = \sqrt{n}(\theta - \hat\T_{\G,n} )$ given $\bX$.
\end{theorem}

Theorem \ref{thm2} shows that the DPD-based posterior distribution satisfies a Bernstein--von Mises type theorem. In particular, this result implies that the EDPE is asymptotically equivalent to the MDPDE and inherits the same asymptotic distribution, as formalized in the next theorem.

\begin{theorem}\label{thm3}
In addition to the assumptions of Theorem \ref{thm2}, assume that $\int \|\T\|\pi(\T)\,d\T <\infty$. Then, it holds that
\[\sqrt{n}(\hat\T_{\G,n}^{EDPE} - \hat\T_{\G,n})\ \stackrel{a.s.}{\longrightarrow}\ 0.\]
Furthermore, if assumptions {\bf A3}–{\bf A6} are also satisfied, we obtain
\[\sqrt{n}(\hat\T_{\G,n}^{EDPE}-\T_0)\ \stackrel{d}{\longrightarrow}\ N_d(0, \mathcal{J}^{-1}_{\G} \mathcal{I}_{\G}\mathcal{J}^{-1}_{\G}).\]
\end{theorem}

\subsection{Application to GARCH models}
In this subsection, we verify that the conditions and assumptions introduced in the previous subsection are satisfied by the GARCH$(p,q)$ model under standard regularity conditions. Consider the following  GARCH($p,q$) model:
\begin{equation}\label{GARCH}
\begin{split}
X_t &= \sigma_t(\T)\,\epsilon_t,\quad \ep_t\sim\ i.i.d.\ N(0,1),\\
\sigma_t^2(\T) &= \omega + \sum_{i=1}^p \A_i X_{t-i}^2
+ \sum_{j=1}^q \B_j \sigma_{t-j}^2(\T),
\end{split}
\end{equation}
where $\T = (\omega, \A_1, \ldots, \A_p, \B_1, \ldots, \B_q)'$ with $\omega > 0$ and  $\A_i, \B_j \ge 0$. We assume that the process $\{X_t \mid t \in \mathbb{Z}\}$ generated from \eqref{GARCH} with the true parameter $\theta_0$ is strictly stationary and ergodic. It is well known for GARCH models that strict stationarity and ergodicity hold if and only if the top Lyapunov exponent is strictly negative; see, for example, \cite{francq:zakoian:2004}. We further assume that the parameter space $\Theta$ is a compact subset of $(0,\infty)\times[0,\infty)^{p+q}$. 

Let $\{\tilde \sigma_t^2(\T)\}_{t=1}^n$ be the proxy process for the conditional variance process $\{\sigma_t^2(\T)\}_{t=1}^n$, defined recursively by
\begin{eqnarray}\label{tsig}
    \tsig_{t}^2(\T) = \W+\sum_{i=1}^{p}\A_i X_{t-i}^2+\sum_{j=1}^{q}\B_j\tsig_{t-j}^2(\T),
\end{eqnarray}
with suitably chosen initial values for $X_0^2=\cdots=X_{1-p}^2$ and $\tsig_{0}^2=\cdots=\tsig_{1-q}^2$.
Since
$$ \tilde{f}_{\T}(x|\mF_{t-1}) =  \frac{1}{\sqrt{2\pi}\tsig_{t}(\T)} \exp\bigg(-\frac{x^2}{2\tsig_{t}^2(\T)}\bigg),$$
the objective function \eqref{Q} for the model \eqref{GARCH} is given by 
\begin{equation}\label{obj_GARCH}
\begin{aligned}
\widetilde Q_n^{(\G)}(\T)
&=
\left\{
\begin{array}{ll}
\displaystyle
\sum_{t=1}^n
\bigg(\frac{1}{\sqrt{2\pi}\,\tsig_t(\T)}\bigg)^\G
\bigg\{\frac{1}{\G}\exp\!\left(-\dfrac{\G X_t^2}{2\tsig_t^{2}(\T)}\right)
-\Big(\frac{1}{1+\G}\Big)^{\frac{3}{2}}\bigg\}
& \mbox{, $\G>0$,}\\[20pt]
\displaystyle
\sum_{t=1}^n \bigg(-\log \sqrt{2\pi} \tsig_t(\T)-\frac{X_t^2}{2\tsig_t^2(\T)}\bigg)
& \mbox{, $\G=0$.}
\end{array}
\right.
\\[5pt]
&:= \sum_{t=1}^n \tilde q_t^{(\gamma)}(\T).
\end{aligned}
\end{equation}
The corresponding DPD-based posterior is then
\begin{align*}
   \pi_{\G}(\theta\mid\bX)
      & \propto \exp \big(\widetilde Q_n^{(\G)}(\T)  \big) \pi(\theta).
\end{align*}

For Theorem \ref{thm:gamma0} to hold for the GARCH model above, we now check that the conditions concerned with the model are satisfied. Conditions {\bf C1} and {\bf C2} are trivially satisfied, and, in view of Remark \ref{rmk:C3}, a shifted version of condition {\bf C3} is shown to hold by Lemma \ref{lem:GARCH_C3}. Hence, we have the following result.

\begin{prop}
If a prior $\pi(\theta)$ satisfies conditions {\bf C4} and {\bf C5}, then the conclusions of Theorem~\ref{thm:gamma0} hold for the GARCH model defined in \eqref{GARCH}.
\end{prop}

The standard conditions usually imposed for the GARCH model above are as follows:
 \begin{enumerate}
\item[\bf G1.] $\displaystyle \sup_{\T\in\Theta}\sum_{j=1}^{q}\beta_j <1.$
\item[\bf G2.] If $q>0\,$, ${\mathcal{A}}_{\theta_0}(z)$ and
${\mathcal{B}}_{\theta_0}(z)$ have no common root,
${\mathcal{A}}_{\theta_0}(1) \neq 1$, and $\alpha_{0p}+\beta_{0q}
\neq 0$, where ${\mathcal{A}}_{\T}(z)=\sum_{i=1}^p \A_i\,z^i$
and ${\mathcal{B}}_{\T}(z)=1-\sum_{j=1}^q\beta_j\,z^j$.
(Conventionally, ${\mathcal{A}}_{\T}(z)=0$ if $p=0$ and
${\mathcal{B}}_{\T}(z)=1$ if $q=0$.)
\item[\bf G3.] $\theta_0$ is in the interior of $\Theta$.
\end{enumerate}

Theorem \ref{thm1} for the GARCH model above has already been established in \cite{lee2009minimum}. Thus, it remains to verify that assumptions {\bf A1--A7} in Subsections \ref{sub:asymp} are satisfied under conditions {\bf G1--G3} in order to apply Theorems \ref{thm2} and \ref{thm3}. 

The function $q_t^{(\gamma)}(\T)$ corresponding to $\tilde q_t^{(\gamma)}(\T)$ is defined by replacing $\tilde \sigma_t(\T)$ with $\sigma_t(\T)$ in \eqref{GARCH}. It is straightforward to see that  $q_t^{(\gamma)}(\T)$ is continuous in $\T$. Condition {\bf G1} ensures stationarity and ergodicity of the process $\{\sigma_t(\theta)\mid t\in\mathbb{Z}\}$, which implies the second part of assumption {\bf A1}. Moreover, noting that
\[ \big|q_t^{(\gamma)}(\T)\big| \le \left(\frac{1}{\sqrt{2\pi}}\right)^\gamma \left\{\frac{1}{\gamma}+\Big(\frac{1}{1+\gamma}\Big)^{3/2}\right\} \sup_{\T\in\Theta}\left(\frac{1}{\omega}\right)^{\gamma/2}\]
and that $\omega$ is bounded away from zero due to the assumption of the compactness of $\Theta\subset(0,\infty)\times[0,\infty)^{p+q}$, we can see that the first part of assumption {\bf A2} holds. The second part is proved in \cite{lee2009minimum} under conditions {\bf G1} and {\bf G2}. With the additional condition {\bf G3}, it is shown in \cite{lee2009minimum} and \cite{song2021test} that assumptions{\bf A3--A5} and  {\bf A6} hold, respectively.
Finally, assumption {\bf A7} is established in Lemma 12 of \cite{song2021sequential}. We therefore obtain the following proposition.

\begin{prop}\label{prop:GARCH_asymp}
Suppose that the GARCH$(p,q)$ model \eqref{GARCH} satisfies conditions {\bf G1--G3}. If a prior $\pi(\theta)$ is positive and continuous at $\T_0$, then
the asymptotic results stated in Theorems \ref{thm2} and  \ref{thm3} hold for the GARCH$(p,q)$ model.
\end{prop}

The posterior distribution induced by the DPD-based objective function is analytically intractable, in that posterior quantities of interest cannot be evaluated in closed form. Posterior inference is therefore carried out using Markov chain Monte Carlo (MCMC) methods. Specifically, we recommend using the Hamiltonian Monte Carlo (HMC) algorithm with its adaptive variant, the No-U-Turn Sampler, to generate samples from the DPD-based posterior distribution.

The HMC exploits gradient information of the log-posterior, $\nabla_{\theta} \log \pi_{\gamma}(\theta \mid \bX)$, to efficiently explore the target distribution by simulating Hamiltonian dynamics (see \cite{hoffman2014no}). This approach is particularly well suited for GARCH-type models, where the conditional variance recursions \eqref{tsig} induce complex and nonlinear dependencies among the parameters, which may lead to slow mixing in standard MCMC schemes such as random-walk Metropolis--Hastings.

Let $\{\T_{\G}^{(1)},\ldots,\T_{\G}^{(m)}\}$ denote posterior samples generated by the HMC algorithm. These samples are used to approximate posterior quantities of interest, including Bayes estimators such as the posterior mean
\[
\hat{\T}_{\G,n}^{EDPE}
= \int \T \pi_{\G}(\T \mid \mathbf{X}_n)\, d\T
\approx \frac{1}{m}\sum_{i=1}^m \T_{\G}^{(i)}.
\]
The resulting posterior samples and their Monte Carlo summaries are subsequently used for posterior inference, convergence diagnostics, and performance evaluation in the following simulation study.

\section{Simulation study}\label{Simulation study}
In this section, we investigate the performance of the proposed robust Bayesian procedure in comparison with the conventional Bayesian method. Specifically, we examine the finite-sample behavior of the EDPE and the EOPE under the following GARCH$(1,1)$ model:
\begin{equation}\label{garch11}
\begin{split}
X_t &= \sigma_t(\T)\,\epsilon_t, \qquad \epsilon_t \sim i.i.d.\ N(0,1),\\
\sigma_t^2(\T) &= \W + \A X_{t-1}^2 + \B \sigma_{t-1}^2(\T),
\end{split}
\end{equation}
where $\T=(\W,\A,\B)$ with $\W>0$ and $\A,\B\ge0$.

\begin{table}[t]
\tabcolsep=5pt
\renewcommand{\arraystretch}{1.4}
\centering
{\footnotesize
\caption{Sample means of the EOPEs and the EDPEs with total scaled RMSEs in parentheses for $\theta_1=(1,0.2,0.4)$ under the uncontaminated setting}
\label{tab:uncontaminated_theta1}
\begin{tabular}{ccccccccccc}
\hline
 & &   & \multicolumn{8}{c}{EDPE $(\gamma)$} \\
\cmidrule(ll){4-10}
$n$ & $\theta_1$ 
& EOPE 
& 0.05 & 0.10 & 0.20 & 0.30 & 0.50 & 0.75 & 1.00 \\
\hline
\multirow{4}{*}{500}
& $\omega$ & 1.187 & 1.194 & 1.201 & \textcolor{red}{1.215} & 1.224 & 1.234 & 1.245 & 1.265 \\
& $\alpha$ & 0.217 & 0.220 & 0.223 & \textcolor{red}{0.231} & 0.240 & 0.265 & 0.298 & 0.326 \\
& $\beta$  & 0.318 & 0.314 & 0.311 & \textcolor{red}{0.304} & 0.299 & 0.291 & 0.286 & 0.285 \\
&           & (0.975) & (0.956) & (0.941) & \textcolor{red}{(0.932)} & (0.939) & (0.991) & (1.109) & (1.228) \\
\hline
\multirow{4}{*}{1000}
& $\omega$ & 1.127 & 1.141 & \textcolor{red}{1.154} & 1.178 & 1.195 & 1.224 & 1.236 & 1.248 \\
& $\alpha$ & 0.208 & 0.210 & \textcolor{red}{0.212} & 0.215 & 0.220 & 0.233 & 0.259 & 0.290 \\
& $\beta$  & 0.346 & 0.340 & \textcolor{red}{0.334} & 0.323 & 0.315 & 0.302 & 0.294 & 0.288 \\
&           & (0.747) & (0.738) & \textcolor{red}{(0.736)} & (0.743) & (0.757) & (0.809) & (0.900) & (1.034) \\
\hline
\multirow{4}{*}{2000}
& $\omega$ & \textcolor{red}{1.080} & 1.090 & 1.103 & 1.127 & 1.152 & 1.191 & 1.219 & 1.233 \\
& $\alpha$ & \textcolor{red}{0.205} & 0.206 & 0.207 & 0.210 & 0.213 & 0.220 & 0.234 & 0.257 \\
& $\beta$  & \textcolor{red}{0.365} & 0.361 & 0.355 & 0.344 & 0.332 & 0.314 & 0.301 & 0.293 \\
&           & \textcolor{red}{(0.569)} & (0.574) & (0.581) & (0.602) & (0.629) & (0.683) & (0.759) & (0.860) \\
\hline
\end{tabular}}
\end{table}

We specify a prior distribution for $\T$ such that the joint prior factorizes as
\[
\pi(\T)=\pi(\W)\,\pi(\A,\B).
\]
Although compactness of the parameter space is assumed for theoretical development in the preceding section, we adopt a truncated normal prior for $\omega$ in this simulation study:
\[
\omega \sim N_{(0,\infty)}(\mu_w,\sigma_w^2),
\]
where $\mu_w$ and $\sigma_w^2$ denote fixed hyperparameters. This prior can be regarded as a weakly informative prior with effectively compact support, since the Gaussian tail decays rapidly and assigns negligible prior mass to extremely large values.

For the ARCH and GARCH coefficients $\A$ and $\B$, rather than relying on the general stationarity and ergodicity condition based on the negativity of the Lyapunov exponent, we impose, for simplicity, the sufficient condition $\A+\B<1$, which guarantees both strict stationarity and ergodicity of the GARCH$(1,1)$ process. Accordingly, we assign a uniform prior over the stationary region,
\[
\pi(\A,\B)\propto I\big(\A\ge0,\ \B\ge0,\ \A+\B<1\big),
\]
where $I(\cdot)$ denotes the indicator function.
Combining these choices, the joint prior distribution is given by
\[
\pi(\T)\propto \exp\!\bigg(-\frac{(\W-\mu_w)^2}{2\sigma_w^2}\bigg)
I(\W>0)\,I(\A\ge0,\B\ge0,\A+\B<1).
\]

\begin{table}[!t]
\tabcolsep=5pt
\renewcommand{\arraystretch}{1.4}
\centering
{\footnotesize
\caption{Sample means of the EOPEs and the EDPEs with total scaled RMSEs in parentheses for $\theta_2=(1,0.15,0.8)$ under the uncontaminated setting}
\label{tab:uncontaminated_theta2}
\begin{tabular}{cccccccccc}
\hline
\multicolumn{2}{c}{} 
& \multicolumn{1}{c}{} 
& \multicolumn{7}{c}{EDPE ($\G$)} \\ 
\cmidrule(ll){4-10}
$n$ & $\theta_2$
& EOPE
& 0.05 & 0.10 & 0.20 & 0.30 & 0.50 & 0.75 & 1.00 \\ \hline

\multirow{4}{*}{500}
& $\W$
& \tcr{1.469}
& 1.504 & 1.539 & 1.602 & 1.656 & 1.721 & 1.750 & 1.736 \\

& $\A$
& \tcr{0.171}
& 0.175 & 0.179 & 0.188 & 0.198 & 0.222 & 0.256 & 0.291 \\

& $\B$
& \tcr{0.753}
& 0.748 & 0.742 & 0.730 & 0.717 & 0.688 & 0.640 & 0.583 \\

& 
& \tcr{(0.984)}
& (1.003) & (1.026) & (1.100) & (1.193) & (1.398) & (1.680) & (1.957) \\ \hline

\multirow{4}{*}{1000}
& $\W$
& \tcr{1.262}
& 1.297 & 1.335 & 1.414 & 1.495 & 1.623 & 1.719 & 1.749 \\

& $\A$
& \tcr{0.162}
& 0.165 & 0.168 & 0.174 & 0.182 & 0.201 & 0.226 & 0.253 \\

& $\B$
& \tcr{0.774}
& 0.770 & 0.765 & 0.755 & 0.744 & 0.719 & 0.685 & 0.644 \\

& 
& \tcr{(0.644)}
& (0.663) & (0.692) & (0.774) & (0.883) & (1.117) & (1.396) & (1.645) \\ \hline

\multirow{4}{*}{2000}
& $\W$
& \tcr{1.140}
& 1.162 & 1.189 & 1.253 & 1.325 & 1.479 & 1.630 & 1.719 \\

& $\A$
& \tcr{0.156}
& 0.158 & 0.160 & 0.164 & 0.170 & 0.186 & 0.208 & 0.229 \\

& $\B$
& \tcr{0.786}
& 0.784 & 0.781 & 0.773 & 0.764 & 0.742 & 0.713 & 0.683 \\

& 
& \tcr{(0.423)}
& (0.438) & (0.460) & (0.530) & (0.620) & (0.850) & (1.154) & (1.405) \\ \hline

\end{tabular}}
\end{table}

Given the observations $\bX=(X_1,\cdots,X_n)$ generated from \eqref{garch11}, the ordinary posterior based on the above prior  is given by
\begin{align*}
    \pi(\T\mid\bX) 
    & \propto \prod_{t=1}^{n}\frac{1}{\tsig_t(\T)}\exp\bigg(-\frac{X_t^2}{2\tsig_t^2(\T)}\bigg) \pi(\T),
\end{align*}
and the DPD-based posterior is given as follows:
\begin{align*}
    \pi_{\G}(\T\mid\bX) 
        & \propto \prod_{t=1}^{n}\exp \bigg[ \bigg(\frac{1}{\sqrt{2\pi} \tsig_t(\T)}\bigg)^\G \bigg\{\frac{1}{\G} \exp \left( -\dfrac{\G X_t^2}{2\tsig_{t}^{2}(\T)} \right)-\Big(\frac{1}{1+\G}\Big)^{\frac{3}{2}}\bigg\} \bigg] \pi(\T).
\end{align*}
As initial values for the proxy process $\{\tilde\sigma_t^2(\T)\}_{t=1}^n$, we set $X^2_0=\tilde\sigma_t^2(\T)=X_1^2$.

In this simulation, we consider the following two parameter configurations:
\[
\theta_1=(1,0.2,0.4) \quad \text{and} \quad \theta_2=(1,0.15,0.8),
\]
where $\A+\B$ in $\theta_1$ corresponds to a moderately persistent GARCH process, while $\theta_2$ represents a highly persistent process with $\A+\B$ close to one, a feature commonly observed in empirical applications. The sample sizes $n$ are set to $500$, $1000$, and $2000$, and we generate sample paths of length $n+1000$ and discard the first $1000$ observations to mitigate initialization effects.

\begin{figure}[!t]
    \centering
    \includegraphics[width=1\linewidth]{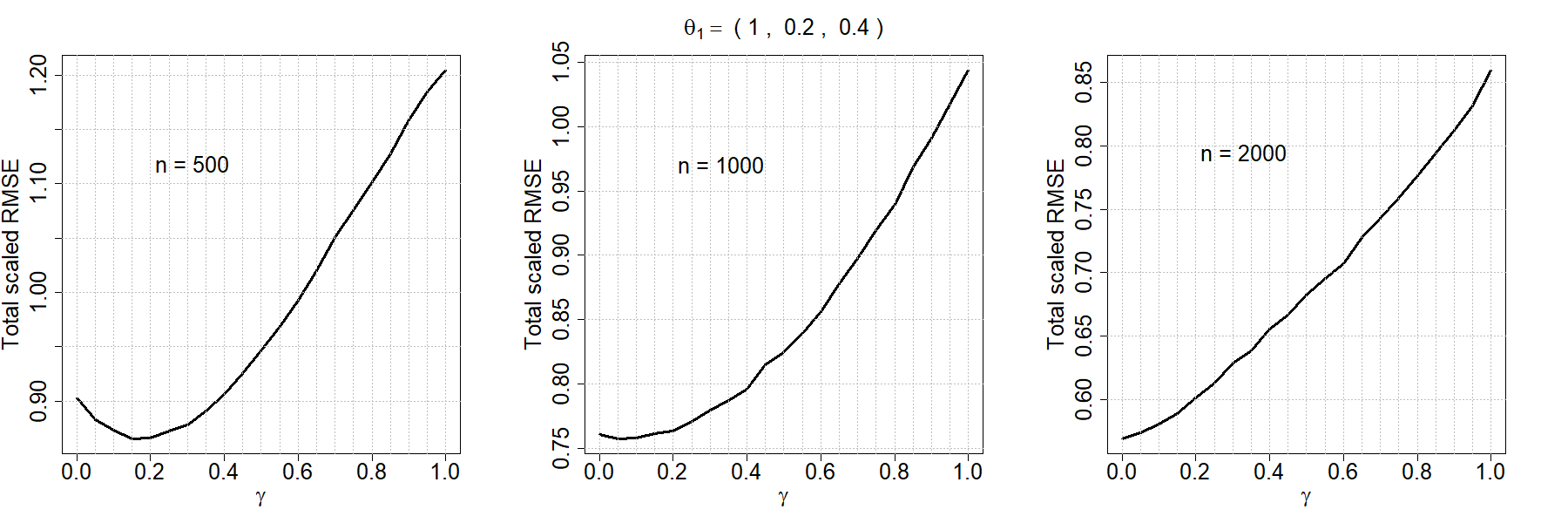}
     \includegraphics[width=1\linewidth]{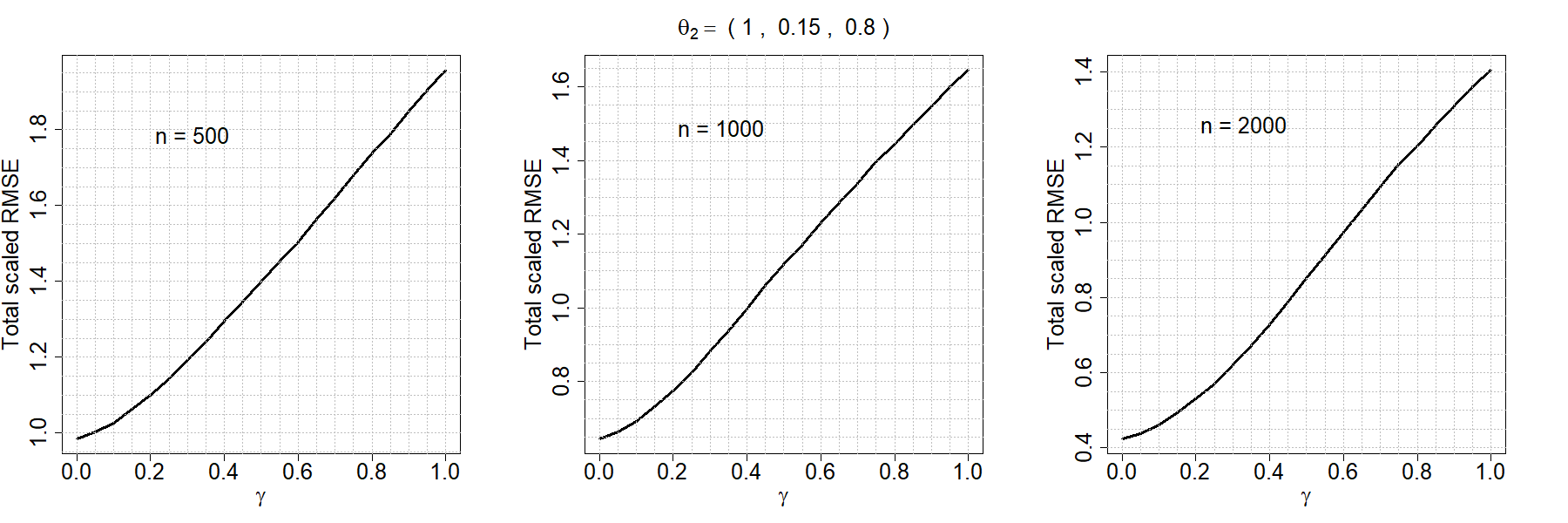}
    \caption{Total scaled RMSE of posterior mean with respect to $\G$ under the uncontaminated setting}
    \label{fig:mse_wrt_alpha_par1_no}
\end{figure}

Recall that the EOPE and the EDPE are defined as the posterior means under the ordinary and DPD-based posterior distributions, respectively. Since these posterior expectations do not admit closed-form expressions, they are approximated by the empirical means of the corresponding MCMC samples.
For the reasons stated in the previous subsection, posterior inference is carried out using HMC with the adaptive No-U-Turn Sampler. For each sample path,  four parallel chains are run with $500$ warm-up iterations and $1000$ sampling iterations per chain to obtain posterior samples, from which the EOPE and EDPE are computed. All computations are performed using the \texttt{Stan} probabilistic programming language via the \texttt{cmdstanr} interface in \texttt{R}. Convergence is assessed using the potential scale reduction factor $\hat R$ and the effective sample size, and all chains exhibit satisfactory convergence across the considered settings. 

For each simulation setting, this procedure is repeated $R=200$ times. The averages of the resulting EOPEs and EDPEs are reported in
Tables~\ref{tab:uncontaminated_theta1} and \ref{tab:uncontaminated_theta2}. To compare the performance of the estimators, we employ the total scaled RMSE,
defined as in \citet{bai2014efficient},
\[
\text{Total scaled RMSE}
= \sum_{j=1}^{3}\frac{1}{|\T_{0j}|} \sqrt{\frac{1}{R}\sum_{r=1}^{R}\Big(\hat{\T}_j^{(r)}-\T_{0j}\Big)^2},
\]
where $\T_{0j}$ denotes the $j$th component of the true parameter vector $\T_0$, and $\hat{\theta}_j^{(r)}$ denotes the Monte Carlo estimate of the posterior mean of the $j$th parameter based on the $r$th replication.  As advocated by \citet{bai2014efficient}, this criterion provides a single, scale-free summary measure of overall estimation accuracy, making it particularly useful when comparing estimators in models involving parameters with substantially different magnitudes.

\begin{table}[!t]
\tabcolsep=5pt
\renewcommand{\arraystretch}{1.4}
\centering
{\footnotesize
\caption{Sample means of the EOPEs and the EDPEs with total scaled RMSEs in parentheses for $\theta_1=(1,0.2,0.4)$ under the  1\% contamination  setting}
\label{table:theta1_IO_t}
\begin{tabular}{cccccccccccc}
\hline
\multicolumn{2}{c}{} 
& \multicolumn{3}{c}{EOPE ($f_\ep$)} 
& \multicolumn{7}{c}{EDPE ($\G$)} \\ 
\cmidrule(lr){3-5}\cmidrule(lr){6-12}
$n$ & $\theta_1$
& $N(0,1)$ & $t(5)$ & $t(7)$
& 0.05 & 0.10 & 0.20 & 0.30 & 0.50 & 0.75 & 1.00 \\ \hline

\multirow{4}{*}{500}
& $\W$
& 1.431 & 1.441 & 1.339
& 1.316 & 1.247 & 1.204 & \tcr{1.211} & 1.235 & 1.261 & 1.286 \\

& $\A$
& 0.301 & 0.295 & 0.272
& 0.271 & 0.253 & 0.244 & \tcr{0.249} & 0.273 & 0.306 & 0.332 \\

& $\B$
& 0.356 & 0.351 & 0.346
& 0.344 & 0.334 & 0.320 & \tcr{0.309} & 0.295 & 0.287 & 0.285 \\

& & (1.894) & (1.430) & (1.265)
& (1.421) & (1.120) & (0.944) & \tcr{(0.940)} & (1.022) & (1.162) & (1.279) \\ \hline

\multirow{4}{*}{1000}
& $\W$
& 1.410 & 1.385 & 1.281
& 1.260 & 1.180 & \tcr{1.139} & 1.150 & 1.192 & 1.232 & 1.259 \\

& $\A$
& 0.287 & 0.280 & 0.257
& 0.255 & 0.236 & \tcr{0.226} & 0.229 & 0.244 & 0.271 & 0.302 \\

& $\B$
& 0.367 & 0.372 & 0.369
& 0.366 & 0.361 & \tcr{0.349} & 0.337 & 0.316 & 0.299 & 0.290 \\

& & (1.597) & (1.156) & (0.979)
& (1.105) & (0.828) & \tcr{(0.693)} & (0.698) & (0.783) & (0.938) & (1.095) \\ \hline

\multirow{4}{*}{2000}
& $\W$
& 1.391 & 1.370 & 1.262
& 1.239 & 1.152 & \tcr{1.103} & 1.110 & 1.156 & 1.208 & 1.239 \\

& $\A$
& 0.272 & 0.271 & 0.248
& 0.244 & 0.227 & \tcr{0.216} & 0.217 & 0.227 & 0.244 & 0.269 \\

& $\B$
& 0.382 & 0.379 & 0.378
& 0.378 & 0.374 & \tcr{0.365} & 0.355 & 0.333 & 0.310 & 0.297 \\

& & (1.333) & (1.028) & (0.838)
& (0.916) & (0.680) & \tcr{(0.551)} & (0.554) & (0.637) & (0.772) & (0.919) \\ \hline

\end{tabular}}
\end{table}

Tables \ref{tab:uncontaminated_theta1} and \ref{tab:uncontaminated_theta2} report the simulation results under the uncontaminated setting for $\theta_1$ and $\theta_2$, respectively. 
For $\theta_2$, the EOPE attains the smallest total scaled RMSE across all considered sample sizes, which is consistent with the efficiency of standard Bayesian inference under correct model specification. 
For $\theta_1$, however, the minimum total scaled RMSE occurs at a small positive tuning parameter in finite samples: it is achieved at $\gamma=0.2$ when $n=500$ and at $\gamma=0.1$ when $n=1000$, while the EOPE becomes best at $n=2000$. 
Overall, the total scaled RMSE tends to increase with $\gamma$, but the $\theta_1$ results indicate that, at smaller $n$, the EDPE with a small $\gamma$ can perform comparably to (or slightly better than) the EOPE, which may reflect finite-sample and Monte Carlo variability rather than a genuine efficiency improvement. As $n$ increases from 500 to 2000, the RMSE values decrease for all estimators, and the advantage shifts toward the EOPE, supporting the expected asymptotic efficiency under the true model. Figure \ref{fig:mse_wrt_alpha_par1_no} visually corroborates these patterns.

\begin{table}[!t]
\tabcolsep=5pt
\renewcommand{\arraystretch}{1.4}
\centering
{\footnotesize

\caption{Sample means of the EOPEs and the EDPEs with total scaled RMSEs in parentheses for $\theta_2=(1,0.15,0.8)$ under the  1\% contamination  setting}
\label{table:theta2_IO_t}
\begin{tabular}{cccccccccccc}
\hline
\multicolumn{2}{c}{} 
& \multicolumn{3}{c}{EOPE} 
& \multicolumn{7}{c}{EDPE ($\G$)} \\ 
\cmidrule(lr){3-5}\cmidrule(lr){6-12}
$n$ & $\theta_2$
& $N(0,1)$ & $t(5)$ & $t(7)$
& 0.05 & 0.10 & 0.20 & 0.30 & 0.50 & 0.75 & 1.00 \\ \hline

\multirow{4}{*}{500}
& $\W$
& 1.918 & 1.857 & 1.718
& 1.662 & 1.544 & \tcr{1.497} & 1.533 & 1.619 & 1.668 & 1.668 \\

& $\A$
& 0.189 & 0.196 & 0.195
& 0.186 & 0.184 & \tcr{0.186} & 0.195 & 0.220 & 0.256 & 0.295 \\

& $\B$
& 0.774 & 0.776 & 0.769
& 0.769 & 0.763 & \tcr{0.751} & 0.738 & 0.706 & 0.653 & 0.588 \\

& & (1.698) & (1.361) & (1.261)
& (1.279) & (1.064) & \tcr{(0.955)} & (1.015) & (1.250) & (1.585) & (1.911) \\ \hline

\multirow{4}{*}{1000}
& $\W$
& 1.798 & 1.699 & 1.532
& 1.506 & 1.376 & \tcr{1.328} & 1.374 & 1.514 & 1.631 & 1.679 \\

& $\A$
& 0.188 & 0.195 & 0.190
& 0.182 & 0.176 & \tcr{0.175} & 0.181 & 0.200 & 0.226 & 0.255 \\

& $\B$
& 0.783 & 0.784 & 0.781
& 0.781 & 0.777 & \tcr{0.770} & 0.760 & 0.735 & 0.699 & 0.653 \\

& & (1.432) & (1.144) & (0.977)
& (1.004) & (0.782) & \tcr{(0.673)} & (0.731) & (0.976) & (1.287) & (1.576) \\ \hline

\multirow{4}{*}{2000}
& $\W$
& 1.620 & 1.538 & 1.367
& 1.346 & 1.234 & \tcr{1.191} & 1.226 & 1.364 & 1.537 & 1.645 \\

& $\A$
& 0.184 & 0.192 & 0.184
& 0.177 & 0.169 & \tcr{0.165} & 0.168 & 0.183 & 0.206 & 0.229 \\

& $\B$
& 0.796 & 0.793 & 0.792
& 0.793 & 0.790 & \tcr{0.784} & 0.778 & 0.759 & 0.729 & 0.696 \\

& & (1.132) & (0.939) & (0.751)
& (0.765) & (0.572) & \tcr{(0.475)} & (0.510) & (0.710) & (1.033) & (1.315) \\ \hline

\end{tabular}}
\end{table}

To examine the robustness of the proposed method against data contamination, we incorporate innovation outliers into the model. Specifically, instead of the standard normal innovations, we generate contaminated innovations according to the following scheme:
\begin{eqnarray*}\label{io}
\epsilon_t = \epsilon_t^{(0)} + s \cdot p_t \cdot \text{sign}(\epsilon_t^{(0)}),
\end{eqnarray*}
where $\epsilon_t^{(0)}$ and $p_t$ are sequences of independent random variables drawn from $N(0,1)$ and $\text{Bernoulli}$ (0.01), respectively. The scaling factor is set to $s=5$, which determines the magnitude of the contamination. Consequently, with a probability of $1\%$, the innovation is shifted by $+5$ or $-5$ in the direction of its sign, mimicking extreme shocks to the system. 

In the presence of such heavy-tailed data or outliers, it is a common practice in Bayesian inference to adopt a likelihood function based on a heavy-tailed distribution, such as the Student's $t$-distribution, rather than the Gaussian distribution. To facilitate a comprehensive comparison, We compute the EOPEs using Student's $t$ likelihoods with 5 and 7 degrees of freedom, in addition to the standard Gaussian EOPE.

\begin{figure}[!t]
    \centering
    \includegraphics[width=1\linewidth]{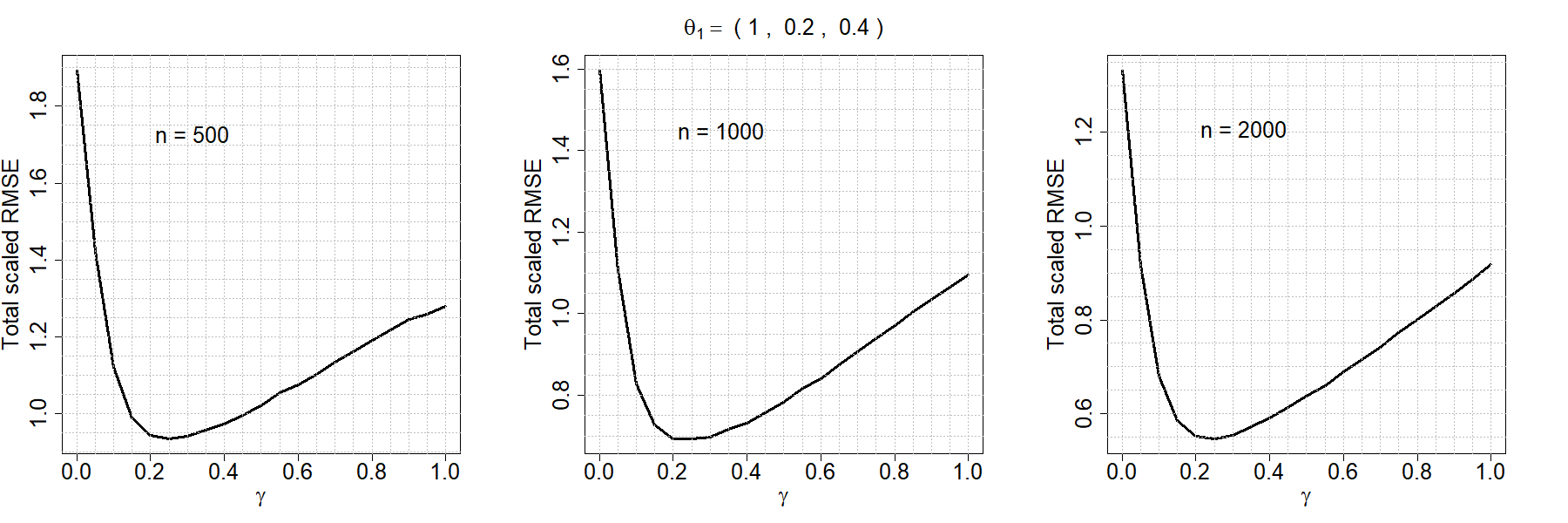}
    \includegraphics[width=1\linewidth]{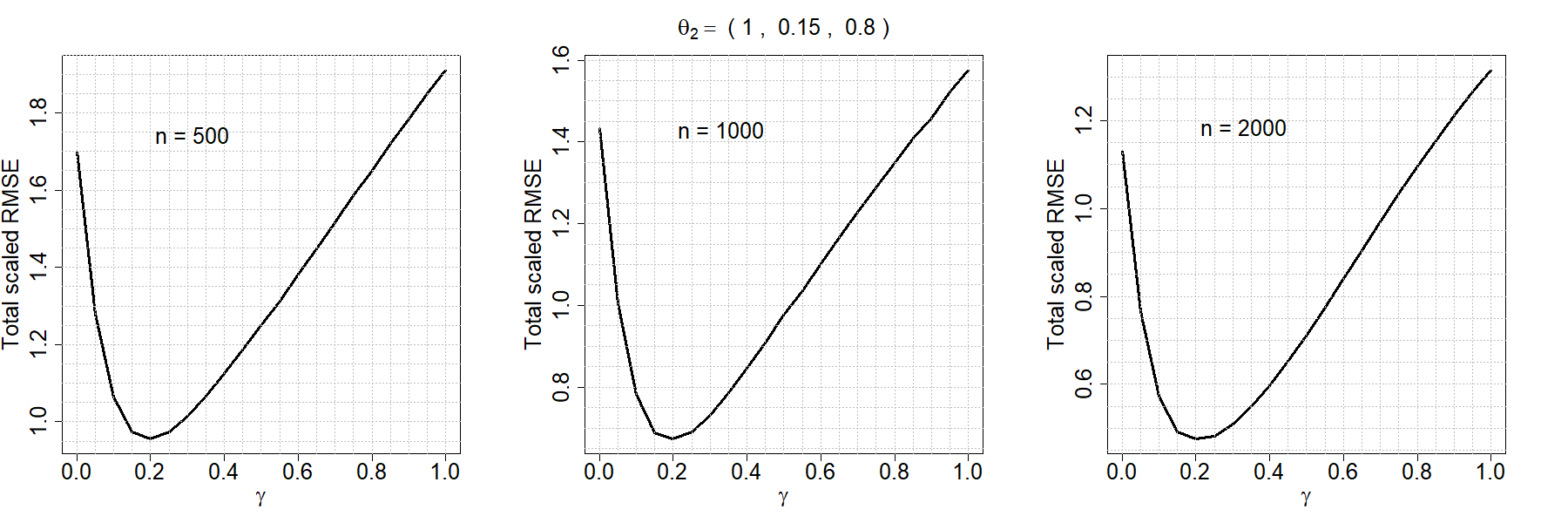}
\caption{Total scaled RMSE of posterior mean with respect to $\G$ under the 1\% contamination  setting}
    \label{fig:mse_wrt_alpha_par1_io}
\end{figure}

Tables \ref{table:theta1_IO_t} and \ref{table:theta2_IO_t} summarize the simulation results for $\theta_1$ and $\theta_2$ under the $1\%$  contamination. Several key observations can be made. First, the performance of the standard Gaussian EOPE ($N(0,1)$) deteriorates significantly compared to the uncontaminated case, yielding the highest total scaled RMSEs across all settings. Second, while the EOPEs based on Student's $t$-distributions provide a marked improvement over the Gaussian estimator by accommodating the heavy tails, they are still outperformed by the proposed DPD-based estimators. Specifically, the EDPEs with tuning parameters around $\gamma=0.2$ to $\gamma=0.3$ consistently achieve the lowest RMSE values (highlighted in red). This indicates that the proposed method offers superior robustness against outliers compared to the conventional strategy of simply switching to a $t$-distribution. Figure \ref{fig:mse_wrt_alpha_par1_io} further illustrates this advantage; unlike the monotonic increase observed in the uncontaminated case, the RMSE curves now exhibit a convex (U-shaped) pattern, suggesting that a small but positive $\gamma$ can best balance robustness and efficiency under contamination.

In summary, the simulation study highlights the practical utility of the proposed robust Bayesian method. Under the uncontaminated setting, the DPD-based estimator with a small $\gamma$ exhibits performance comparable to the efficient EOPE, particularly as the sample size increases. However, in the presence of outliers, the standard EOPE fails to provide reliable estimates, and even the strategy of adopting heavy-tailed likelihoods (e.g., Student's $t$) yields suboptimal results compared to the proposed method. The EDPE with a properly chosen tuning parameter (e.g., $\gamma \approx 0.2$) successfully achieves a balance between robustness and efficiency, offering a substantial reduction in RMSE by effectively down-weighting the influence of outliers. These findings confirm that the proposed method serves as a safe and robust alternative for GARCH modeling, especially when the data appear to contain outliers or anomalies.

\section{Real data analysis}\label{Real data analysis}

In this section, we analyze Bitcoin price data to illustrate the proposed method. Let $\{X_t\}$ denote the daily closing prices and let $\{r_t\}$ be the corresponding log-return series, defined by $r_t = 100 \log(X_t/X_{t-1})$. The dataset covers the period from September 2021 to December 2024 (totally 1,218 observations), as displayed in Figure \ref{fig:Data}. We use the returns from September 2021 to August 2024 for model estimation and reserve the remaining observations for out-of-sample evaluation.

\begin{figure}[!t]
\centering
\includegraphics[width=0.99\linewidth]{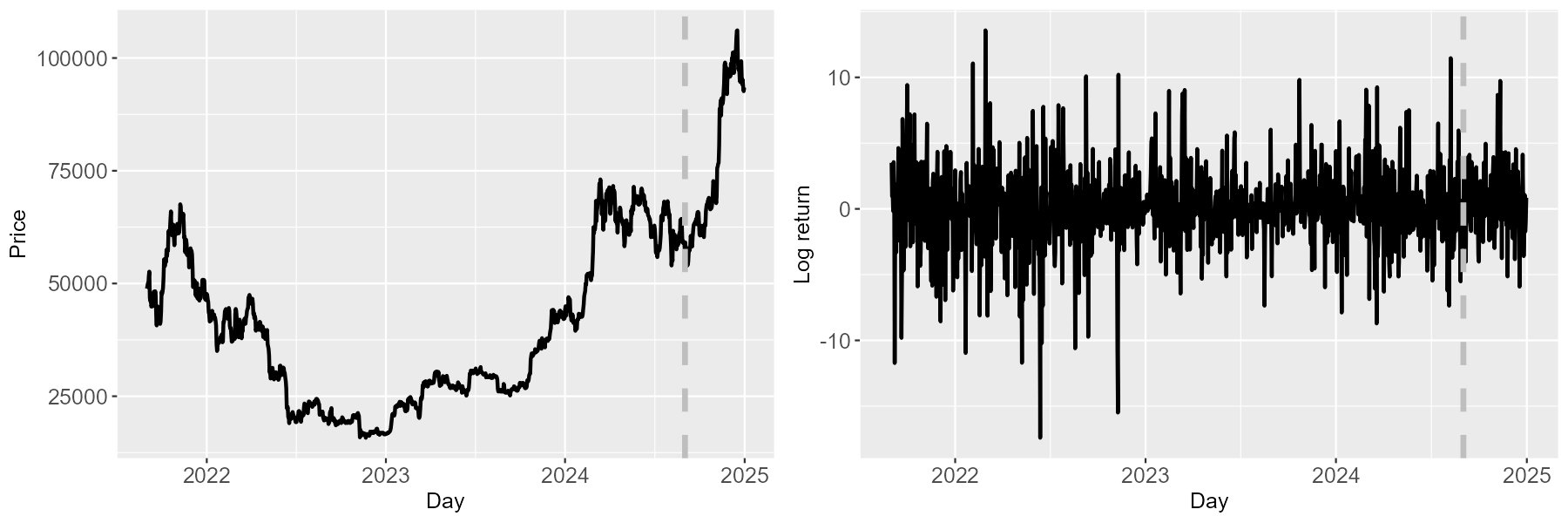}
\caption{Daily closing prices (L) and log returns (R) of BTC-USD from September 2021 to December 2024; the vertical dashed line marks the start of the out-of-sample evaluation period}

\label{fig:Data}
\end{figure}

Preliminary analysis using the Ljung-Box test indicates that, for the in-sample log returns, the smallest $p$-value among lags 1 to 30 is $0.1241$ (at lag 25), suggesting no significant linear serial correlation. In contrast, for the squared log returns, the largest $p$-value is $0.0002$ (at lag 6), providing strong evidence of conditional heteroskedasticity. We also observe typical volatility clustering in the return series. Accordingly, we fit the GARCH$(1,1)$ model defined in \eqref{garch11} with parameter $\T=(\W,\A,\B)$ to the return series $\{r_t\}$, which is a standard and parsimonious specification commonly adopted in empirical volatility modeling.

As shown in the right panel of Figure \ref{fig:Data}, the series also contains several outlying observations that may adversely affect ordinary Bayesian inference. To address this issue, we estimate the model using the EDPE proposed in this study. For comparison, we also consider the conventional EOPE.
The two approaches are evaluated by comparing their one-step-ahead forecasts of the conditional variance and the associated 95\% Value-at-Risk (VaR) over the out-of-sample period. Specifically, the returns from September 2024 to December 2024 (122 days) are used for this out-of-sample analysis.

For each time point $t$, the one-step-ahead conditional variance at time $t+1$ is given by
\begin{align*}
    \hat{\sigma}_{t+1|t}^2 = \hat{\W}_{t} + \hat{\A}_{t}r_{t}^2 + \hat{\B}_t\tilde{\sigma}_t^2,
\end{align*}
where $\hat{\W}_t$, $\hat{\A}_t$, and $\hat{\B}_t$ are the parameter estimates obtained  by either the EOPE or the EDPE using the return series up to time $t$.  The term $\tilde \sigma_t^2$ is computed recursively according to \eqref{tsig} with the corresponding plug-in estimates, using the initial values $r_0^2=\tilde\sigma_0^2=r_1^2$. The forecasting performance is evaluated using the root mean squared error (RMSE) and the mean absolute error (MAE). Since the true conditional variance $\sigma_{t+1}^2$ is unobservable, we use $r_{t+1}^2$ as its proxy, noting that $\E(r_{t+1}^2|\mF_t)=\sigma_{t+1}^2$ under the GARCH model. Accordingly, over the 122 out-of-sample days, the RMSE and MAE are calculated as
$$\text{RMSE}=\sqrt{\frac{1}{122}\sum_{t=1096}^{1217}(\hat{\sigma}_{t+1|t}^2-r_{t+1}^2)^2} \quad\text{and}\quad\text{MAE}= \frac{1}{122}\sum_{t=1096}^{1217}\big|\hat{\sigma}_{t+1|t}^2-r_{t+1}^2\big|,$$
where the forecasting period spans from September 1, 2024 to December 31, 2024.

\begin{table}[!t]
\tabcolsep=5pt
\renewcommand{\arraystretch}{1.6}
\centering
{\small
\caption{GARCH parameter estimates and one-step-ahead forecasting performance for BTC-USD data}
\label{table:result_BTC-USD}
\begin{tabular}{ccccccccc}
\hline
         &  & \multicolumn{7}{c}{EDPE ($\G$)} \\[-0.8ex] \cmidrule(ll){3-9}
         & EOPE       & 0.05   & 0.10   & 0.20   & 0.30   & 0.50   & 0.75   & 1.00  \\[-0ex]
          \hline
$\hat\W$ & 1.859  & 1.403  & 1.078  & 0.726  & 0.738  & 0.990  & 1.194  & 1.306  \\
$\hat\A$ & 0.167  & 0.134  & 0.113  & 0.092  & 0.094  & 0.125  & 0.179  & 0.251  \\
$\hat\B$ & 0.637  & 0.694  & 0.736  & 0.776  & 0.746  & 0.617  & 0.484  & 0.396  \\ 
\hline
RMSE     & 3.553 & 3.531 & 3.519 & \tcr{3.512} & 3.517 & 3.534 & 3.553 & 3.577 \\
MAE      & 7.908  & 7.563  & 7.277  & 6.874  & 6.666  & \tcr{6.515}  & 6.593  & 6.832  \\
\hline
VaR violation rate & 0.016  & 0.016  & 0.016  & \tcr{0.049}  & \tcr{0.049}  & 0.057  & 0.066  & 0.074  \\ \hline
\end{tabular}}
\end{table}

The one-step-ahead 95\% VaR forecast at time $t+1$ is defined as 
\begin{eqnarray*}
    VaR_{t+1|t} = F_\ep^{-1}(0.05) \, \hat{\sigma}_{t+1|t}
\end{eqnarray*}
where $F_\ep$ is the cumulative distribution function (CDF) of the innovation $\epsilon_t$ in the GARCH model. Since we assume that $\ep_t$ follows a standard normal distribution in this analysis, we backtest the VaR forecasts by computing the empirical violation rate
\begin{eqnarray*}
    \frac{1}{122} \sum_{t=1096}^{1217} I\left(r_{t+1} < \Phi^{-1}(0.05) \, \hat{\sigma}_{t+1|t} \right),
\end{eqnarray*}
where $I(\cdot)$ denotes the indicator function. Under a correctly specified model, the probability that $r_{t+1}$ falls below $VaR_{t+1|t}$ equals $0.05$,
so empirical violation rates close to $0.05$ indicate that the fitted model is well calibrated and reasonably captures the tail behavior of the return distribution.

The first three rows of Table~\ref{table:result_BTC-USD} report the posterior mean estimates of the GARCH$(1,1)$ parameters obtained using the EOPE and the EDPE for different values of the tuning parameter $\gamma$. One can see noticeable differences between the EOPE and EDPE estimates, particularly for the parameter $\W$. Specifically, the EOPE yields substantially larger values than the EDPE across the entire range of $\gamma$.
Recall from the simulation study that, in the absence of outliers, the EOPE and the EDPE produce similar estimates for small $\gamma$. In contrast, in the presence of outliers, the outlier-sensitive EOPE tends to overestimate $\W$, leading to estimates that differ substantially from those of the EDPE, even when $\gamma$ is small. In light of these findings, the inflated estimate of $\W$ produced by the EOPE strongly suggests the  presence of observations in the return series that behave like outliers and exert a non-negligible influence on the EOPE.


\begin{figure}[!t]
    \centering
    \includegraphics[width=0.8\linewidth]{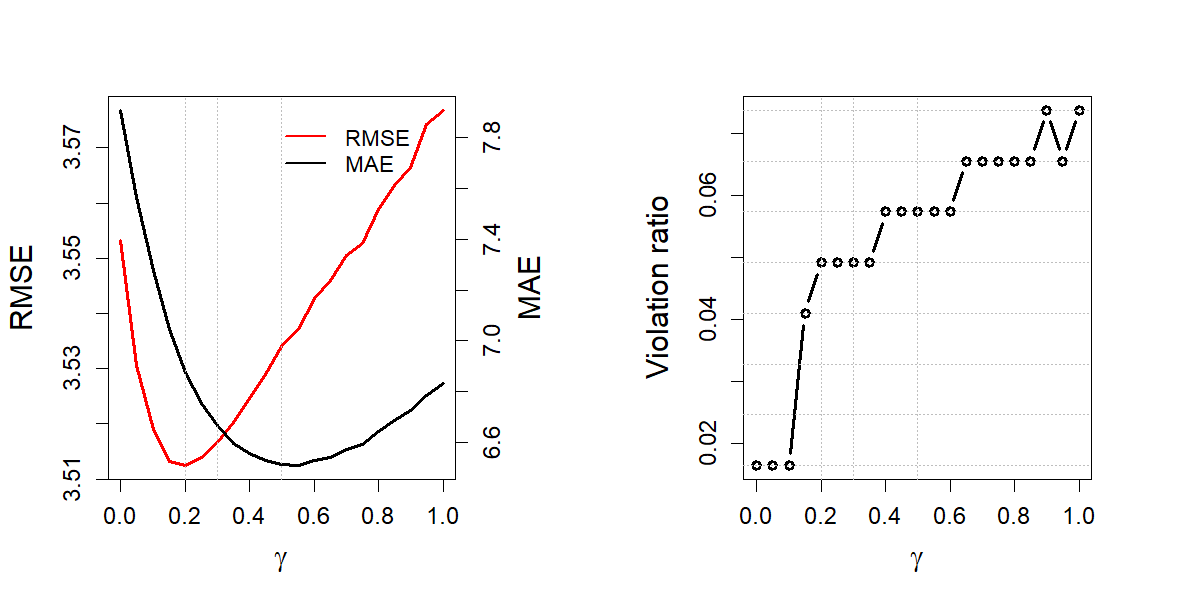}
\caption{RMSE and MAE for conditional variance forecasts (L) and the empirical 95\% VaR violation ratio (R)}
    \label{fig:mse_mae_p}
\end{figure}

The bottom rows of Table \ref{table:result_BTC-USD} report the one-step-ahead forecasting performance in terms of RMSE, MAE, and the empirical violation rate of the 95\% VaR, while Figure \ref{fig:mse_mae_p} illustrates how these measures vary with the tuning parameter $\gamma$ over the range $[0,1]$. As anticipated, the EDPE outperforms the EOPE for most values of $\gamma$ in terms of both RMSE and MAE, indicating that the proposed Bayesian estimator  can serve as a useful alternative when the data may contain observations with outlier-like influence.

We also observe that the RMSE and the MAE attain their minima at $\G=0.2$ and $\G=0.5$, respectively. Although both measures are reported to assess forecasting accuracy, we place greater emphasis on RMSE as it penalizes large forecast errors more heavily than MAE and is therefore more informative for volatility modeling and risk management, where capturing extreme market movements is crucial. This consideration also motivates the data-driven selection rule for $\gamma$ discussed in Remark \ref{rmk:select_gamma}, which is based on minimizing the out-of-sample RMSE. According to this criterion, we select $\gamma=0.2$ as the optimal tuning value for the given dataset. 

Notably, this choice of $\gamma$ is also associated with near-optimal VaR calibration. As seen in the last row of the table, the EOPE yields a violation rate of $0.016$ (2 violations out of 122), which is substantially below the nominal 0.05 level, indicating overly conservative VaR estimates. In contrast, the EDPE with $\G$ between $0.20$ and $0.35$ attains violation rates of $0.049$ (6 violations out of 122), very close to the target level of $0.05$. For larger values of $\G$, the violation rate exceeds $0.05$, implying that the VaR forecasts become too aggressive  and thus underestimate downside risk. Overall, a moderate value of $\G$ (roughly between $0.2$ and $0.5$) appears to provide a favorable trade-off between forecasting accuracy and VaR calibration.

\begin{figure}[!t]
    \centering
     \includegraphics[height=0.45 \textwidth,width=1\textwidth]{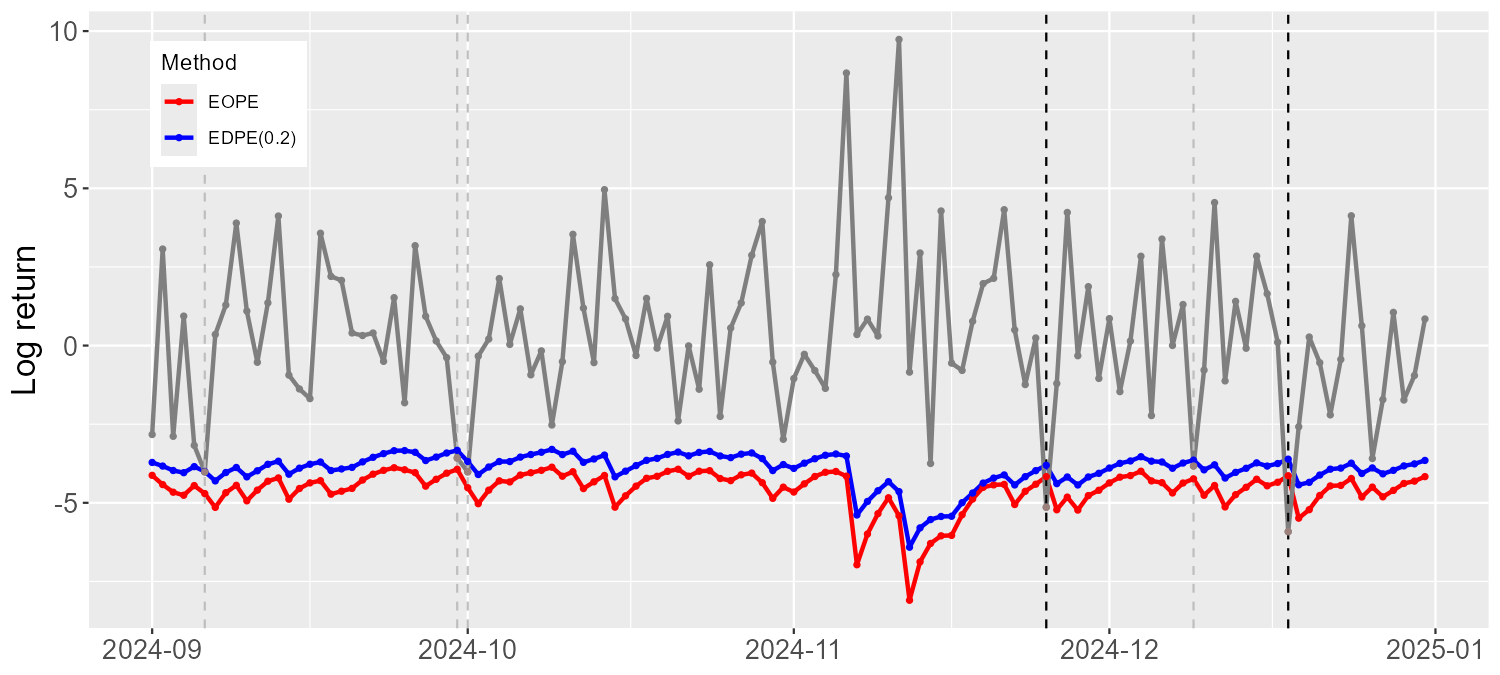}
\caption{Daily log returns and one-step-ahead 95\% VaR forecasts under the EOPE and EDPE for September--December 2024}
    \label{fig:result}
\end{figure}

Figure \ref{fig:result} displays the log returns together with the one-step-ahead 95\% VaR forecasts obtained using the EOPE and the EDPE with $\G=0.2$. The gray line represents the realized return series, while the red and blue lines depict the corresponding VaR forecasts computed by the EOPE and the EDPE, respectively. The black vertical dashed lines indicate dates on which VaR violations occur for both estimators, whereas the gray vertical dashed lines correspond to violations that occur only for the EDPE-based forecasts. We can clearly see that the VaR forecasts produced by the EOPE are relatively wider, which appears to be associated with the relatively large estimate of $\W$, likely influenced by the presence of outlier-like observations as discussed earlier.  In contrast, the EDPE yields VaR forecasts that are neither overly conservative nor overly aggressive, but instead well aligned with the nominal coverage level.

Overall, the real-data analysis of BTC-USD confirms the practical usefulness of the proposed robust Bayesian estimator. Compared with the conventional EOPE, the EDPE yields improved out-of-sample forecasts of the conditional variance over the evaluation period. Moreover, the VaR backtesting results indicate that the EOPE produces overly conservative risk forecasts, whereas the EDPE with a moderate tuning parameter achieves empirical violation rates close to the nominal level, providing better calibration of tail risk. Taken together, these findings suggest that, for financial return series that exhibit occasional extreme shocks, the proposed EDPE offers a favorable trade-off between predictive accuracy and risk calibration, and thus serves as a robust alternative to standard Bayesian inference for GARCH-type models.

\section{Conclusion}\label{Conclusion}
Outliers and abrupt shocks can severely compromise Bayesian inference for conditionally heteroscedastic time series by dominating the likelihood and, in turn, the posterior distribution. To mitigate this sensitivity, we propose a robust Bayesian framework based on a DPD-driven pseudo-posterior indexed by a tuning parameter $\gamma$. The resulting estimator, defined as the posterior mean under the DPD-based posterior, connects smoothly to the ordinary Bayes estimator  as $\gamma\downarrow 0$ and becomes robust for $\gamma>0$ by downweighting extreme observations.

We establish asymptotic theory under stationarity and ergodicity, including a Bernstein--von Mises type approximation for the DPD-based posterior and asymptotic equivalence between the EDPE and the MDPDE. We also show that the assumptions required for these asymptotic results hold for GARCH$(p,q)$ models under the standard regularity conditions commonly imposed in the GARCH literature.  Numerical studies for GARCH$(1,1)$ demonstrate that moderate values of $\gamma$ substantially improve accuracy under contamination and can outperform heavy-tailed likelihood-based alternatives, while remaining competitive in uncontaminated samples. In an application to BTC-USD returns, the EDPE improves one-step-ahead volatility forecasts and yields better-calibrated 95\% VaR, producing empirical violation rates close to the nominal level.

In sum, the proposed approach provides a practical and theoretically grounded robust Bayesian tool for volatility modeling in the presence of occasional extreme shocks. Important extensions include adapting the methodology to integer-valued time series models, where extreme counts and outliers are often consequential, as well as to multivariate Bayesian settings to further broaden its applicability.

\section{Appendix}
In this appendix, we present proofs of the main theorems and related lemmas.\\

\noindent{\bf Proof of Theorem \ref{thm:gamma0}}\\
Let $\Delta_\gamma(\theta):=\widetilde Q_n^{(\gamma)}(\theta)-\sum_{t=1}^n\log \tilde f_\theta(X_t|\mathcal{F}_{t-1}) $.
Then, it follows from condition {\bf C3} that 
\[
e^{-V_n\gamma}\le e^{\Delta_\gamma(\theta)}\le e^{V_n\gamma}
\qquad\text{for all }\theta\in\Theta
\]
and thus, we have
\[
\sup_{\theta\in\Theta}\big|e^{\Delta_\gamma(\theta)}-1\big|
\le \max\{e^{V_n\gamma}-1,\,1-e^{-V_n\gamma}\}
\longrightarrow 0
\quad as\ \gamma\downarrow 0,
\]
which implies
\[
\sup_{\theta\in\Theta}
\left|
\frac{\exp\!\big(\widetilde Q_n^{(\gamma)}(\theta)\big)}{\tilde L(\theta \mid \bX)}-1
\right|
=
\sup_{\theta\in\Theta}\big|e^{\Delta_\gamma(\theta)}-1\big|
\longrightarrow 0\quad as\ \gamma\downarrow 0.
\]
Since $\Theta$ is compact and $\theta\mapsto \tilde L(\theta\mid\bX)$ is continuous on $\Theta$ by condition {\bf C2}, the function $\tilde L(\theta\mid\bX)$ is bounded on $\Theta$. Moreover, since $\pi(\theta)$ is continuous on $\Theta$ by condition {\bf C4}, it is also bounded.
Therefore, we obtain
\begin{equation}\label{unif_conv}
\begin{split}
\sup_{\theta\in\Theta}\left|\exp(\widetilde Q_n^{(\gamma)}(\theta))\pi(\theta)-\tilde L(\theta\mid \bX) \pi(\theta)\right|
&\ \le\ \sup_{\theta\in\Theta}\left|\tilde L(\theta\mid \bX)\pi(\theta)\right|
\sup_{\theta\in\Theta}\left|\frac{\exp(\widetilde Q_n^{(\gamma)}(\theta))}{\tilde L(\theta\mid \bX)}-1\right|\\
&\ \longrightarrow\ 0 \quad \text{as } \gamma\downarrow 0.
\end{split}
\end{equation}


Next, let $u_\gamma(\theta):=\exp(\widetilde Q_n^{(\gamma)}(\theta))\pi(\theta)$ and $u_0(\theta):=\tilde L(\theta\mid\bX)\pi(\theta)$, with normalizing constants
$Z_\gamma:=\int_\Theta u_\gamma(\theta)\,d\theta$ and $Z_0:=\int_\Theta u_0(\theta)\,d\theta$. By \eqref{unif_conv}, we have $\sup_{\theta\in\Theta}|u_\gamma(\theta)-u_0(\theta)|\to0$. Thus, it follows from compactness of $\Theta$ that
\begin{equation}\label{convs}
   |Z_\gamma-Z_0|
\le \int_\Theta |u_\gamma(\theta)-u_0(\theta)|\,d\theta
\le \sup_{\theta\in\Theta}|u_\gamma(\theta)-u_0(\theta)|\int_\Theta 1\,d\theta
\longrightarrow 0 \quad \text{as } \gamma\downarrow 0.
\end{equation}
Moreover, using conditions {\bf C4} and {\bf C5} together with the positivity of $\tilde L(\theta\mid\bX)$, we can show that $Z_0>0$, yielding $Z_\gamma\ge Z_0/2$ for all sufficiently small $\gamma$. Writing the normalized densities as
$p_\gamma(\theta):=u_\gamma(\theta)/Z_\gamma$ and $p_0(\theta):=u_0(\theta)/Z_0$, we obtain
\[
\int_\Theta |p_\gamma(\theta)-p_0(\theta)|\,d\theta
\le \frac{1}{Z_\gamma}\int_\Theta |u_\gamma(\theta)-u_0(\theta)|\,d\theta
+ \left|\frac1{Z_\gamma}-\frac1{Z_0}\right|\int_\Theta u_0(\theta)\,d\theta \longrightarrow 0 \quad \text{as } \gamma\downarrow 0,
\]
since $Z_\gamma\to Z_0$ and $\int_\Theta |u_\gamma-u_0|\to0$ by \eqref{convs}. This implies convergence in total variation and completes the proof.
\hfill{$\Box$}\vspace{0.2cm}\\

\begin{lemma}\label{ULLN} Suppose that assumptions {\bf A1} and {\bf A2} hold. Then, for each $\G>0$, we have that 
\[ \sup_{\T\in\Theta} \Big| \frac{1}{n}\widetilde Q_n^{(\G)}(\T) - Q^{(\G)}(\T)\Big| = o(1)\quad a.s.,\]
where $Q^{(\G)}(\T)=\E q^{(\G)}_t(\T)$.
\end{lemma}
\begin{proof}
By the ergodicity of $\{q_t^{(\G)}(\T)\}$, we have that $\frac{1}{n}Q_n^{(\G)}(\T)$ converges almost surely to $Q^{(\G)}(\T)$ for each $\T \in \Theta$. Hence, applying Theorem 2.7 of \cite{straumann2006quasi} along with the condition $\E \sup_{\T\in\Theta} \big|q_t^{(\G)}(\T)\big|<\infty$, we obtain
\[ \sup_{\T\in\Theta} \Big| \frac{1}{n} Q_n^{(\G)}(\T) - Q^{(\G)}(\T)\Big| = o(1)\quad a.s.\]
Thus, the lemma follows from the second part of assumption {\bf A2}.
\end{proof}

\noindent{\bf Proof of Theorem \ref{thm1}}\\
Note that $d_\gamma\left(f_{\theta_0}(\cdot\mid\mathcal{F}_{t-1}),f_\theta(\cdot \mid\mathcal{F}_{t-1})\right)$ is almost surely nonnegative and  is uniquely minimized at $\T=\T_0$. Hence, its expectation, $\E\,d_\alpha\left(f_{\theta_0}(\cdot\mid\mathcal{F}_{t-1}),f_\theta(\cdot \mid\mathcal{F}_{t-1})\right)$, is also minimized at $\T_0$.  Since
\[ Q^{(\G)}(\T)=-\frac{1}{1+\G}\E\,d_\gamma\left(f_{\theta_0}(\cdot\mid\mathcal{F}_{t-1}),f_\theta(\cdot\mid\mathcal{F}_{t-1})\right)+\frac{1}{\G(1+\G)}\int f_{\T_0}^{1+\G}(x\mid\mathcal{F}_{t-1})dx,\]
it follows that  $Q^{(\G)}(\T)$ attains its unique maximum  at $\T_0$
 Therefore, by Lemma \ref{ULLN}, the strong consistency of the estimator $\hat \T_{\G,n}$ can be established by the standard arguments.

 Next, we prove the asymptotic normality. By assumption {\bf A3} and the central limit theorem for martingale differences, we have 
\begin{eqnarray}\label{CLT}
    \frac{1}{\sqrt{n}} \pa_\T Q_n^{(\G)}(\T_0)\ \stackrel{d}{\longrightarrow}\ N_d(0,\mathcal{I}_\G).
\end{eqnarray}
 Hence, it follows from assumption {\bf A4} that $\frac{1}{\sqrt{n}}\pa_\T\widetilde Q_n^{(\G)}(\T_0)=O_P(1)$.
Furthermore, under assumptions {\bf A5} and {\bf A6}, it can be shown that for any sequence of random vectors  $\{\T_n^*\}$ converging almost surely to $\T_0$, 
\begin{eqnarray}\label{conv.Q2}
   \frac{1}{n}\paa \widetilde Q_n^{(\G)}(\T^*_n)\ \longrightarrow \ -\mathcal{J}_\G\quad a.s.
\end{eqnarray}
 Since $\pa_\T \widetilde Q_n^{(\G)}(\hat\T_{\G,n})=0$, the mean value theorem yields $\pa_\T \widetilde Q_n^{(\G)}(\T_0)=-\paa \widetilde Q_n^{(\G)}(\bar\T_n)(\hat\T_{\G,n}-\T_0)$, where $\bar\T_n$ is a point on the line segment between $\hat\T_{\G,n}$ and $\T_0$. Thus, we can see that 
\begin{eqnarray*}
\sqrt{n} (\hat{\theta}_{\G,n}-\theta_0)
&=& \mathcal{J}_\G^{-1}\frac{1}{\sqrt{n}}\pa_\T \widetilde Q_n^{(\G)}(\T_0)+
\mathcal{J}_\G^{-1}\Big(\frac{1}{n}\paa \widetilde Q_n^{(\G)}(\bar\T_n) +\mathcal{J}_\G\Big)\sqrt{n}(\hat{\T}_{\G,n}-\theta_0)\\
&=&O_P(1)+o_P(1)\sqrt{n} (\hat{\T}_{\G,n}-\T_0),
\end{eqnarray*}
which implies $\sqrt{n} (\hat{\theta}_{\G,n}-\theta_0)=O_P(1)$. Therefore, it follows from assumption {\bf A4} that
\begin{eqnarray*}
\sqrt{n} (\hat\T_{\G,n}-\theta_0)
= \mathcal{J}_\G^{-1}\frac{1}{\sqrt{n}}\pa_\T Q^{(\G)}_n(\theta_0)+o_P(1),
\end{eqnarray*}
which together with (\ref{CLT}) establishes the asymptotic normality of $\sqrt{n} (\hat\T_{\G,n}-\theta_0)$. This completes the proof.
\hfill{$\Box$}\vspace{0.2cm}\\

\begin{lemma}\label{Q.R2} Suppose that assumptions {\bf A1} and {\bf A2} hold.  For any $\delta>0$, there exists a $\ep>0$ such that, for all sufficiently large $n$,  
\[ \sup_{\|\T-\T_0\| >\delta}  \frac{1}{n}\big(\widetilde Q_n^{(\G)}(\T) - \tilde Q_n^{(\G)}(\T_0)\big) < -\ep \quad a.s.\]
\end{lemma}
\begin{proof}
Recall that $Q^{(\G)}(\T)$ has a unique maximum at $\T_0$. Thus, by the continuity of $Q^{(\G)}(\T)$, for any $\delta>0$, one can choose  $\ep>0$ such that 
\begin{eqnarray*}
    \sup_{\|\T-\T_0\|>\delta} Q^{(\G)}(\T) \leq Q^{(\G)}(\T_0) -3\ep.
\end{eqnarray*}
Using  Lemma \ref{ULLN} and the above result, we obtain that, for any $\T$ satisfying $\|\T-\T_0\|>\delta$ and for all sufficiently large $n$, 
\begin{eqnarray}\label{eq.Q}
    \frac{1}{n}\widetilde Q_n^{(\G)}(\T) \leq Q^{(\G)}(\T) +\ep \leq   Q^{(\G)}(\T_0) -2\ep\quad a.s.
\end{eqnarray}
Furthermore, for all sufficiently large $n$, since $\frac{1}{n}\widetilde Q_n^{(\G)}(\T_0) \geq  Q^{(\G)}(\T_0) -\ep\ a.s.$ by Lemma \ref{ULLN}, it follows from (\ref{eq.Q}) that 
\[\frac{1}{n}\widetilde Q_n^{(\G)}(\T)-\frac{1}{n}\widetilde Q_n^{(\G)}(\T_0) \leq -\ep\quad a.s.\]
for all $\T$ with $\|\T-\T_0\|>\delta$ and for all sufficiently large $n$. This completes the proof.
\end{proof}

\noindent{\bf Proof of Theorem \ref{thm2}}\\
Letting 
\[C_n:=\int_{\mathbb{R}^d} \exp\Big( \widetilde Q_n^{(\G)}\Big( \hat\T_{\G,n} + \frac{t}{\sqrt{n}}\Big)- \widetilde Q_n^{(\G)}\big( \hat\T_{\G,n}\big)\Big)\,\pi\Big( \hat\T_{\G,n} + \frac{t}{\sqrt{n}}\Big)dt,\]
the $\G$-posterior density of $t=\sqrt{n}(\T-\hat\T_{\G,n})$ can be expressed as
\[ \pi_\G(t|\bX)=\frac{1}{C_n} \exp\Big( \widetilde Q_n^{(\G)}\Big( \hat\T_{\G,n} + \frac{t}{\sqrt{n}}\Big)- \widetilde Q_n^{(\G)}\big( \hat\T_{\G,n}\big)\Big)\,\pi\Big( \hat\T_{\G,n} + \frac{t}{\sqrt{n}}\Big).\]
 Define 
 \[g_n(t):=\pi\Big( \hat\T_{\G,n} + \frac{t}{\sqrt{n}}\Big)\exp\Big( \widetilde Q_n^{(\G)}\Big( \hat\T_{\G,n} + \frac{t}{\sqrt{n}}\Big)- \widetilde Q_n^{(\G)}\big( \hat\T_{\G,n}\big)\Big) - \pi(\T_0)e^{-\frac{1}{2}t'\mathcal{J}_\G t}.\]
 Then, the theorem follows if we show that 
 \begin{eqnarray}\label{main}
     \int_{\mathbb{R}^d} \big|g_n(t)\big|dt =o_P(1).
 \end{eqnarray}
To show the above, we split $\mathbb{R}^d$ into two regions $\mathcal{R}_1=\{t \in\mathbb{R}^d \mid \|t\| \leq \delta_0 \sqrt{n} \}$ and $\mathcal{R}_2=\{t\in\mathbb{R}^d \mid \|t\| > \delta_0 \sqrt{n} \}$ for some $\delta_0 >0$. 

First, we deal with the integral on $\mathcal{R}_1$. By the Taylor theorem together with $\pa_\T \widetilde Q_n^{(\G)}(\hat\T_{\G,n})=0$, we have 
\begin{eqnarray}\label{Q2}
    \widetilde Q_n^{(\G)}\Big( \hat\T_{\G,n} + \frac{t}{\sqrt{n}}\Big)- \widetilde Q_n^{(\G)}\big( \hat\T_{\G,n}\big)
= \frac{1}{2n}t'\paa\widetilde Q_n^{(\G)} (\bar\T_{\G,n})t,
\end{eqnarray} 
where $\bar\T_{\G,n}$ lies between $\hat\T_{\G,n} + \frac{t}{\sqrt{n}}$ and $\hat\T_{\G,n}$. For any fixed $t$, since $\bar\T_{\G,n}$ converges almost surely to $\T_0$, it follows from (\ref{conv.Q2}) that  $\frac{1}{n} \paa \widetilde Q_n^{(\G)} (\bar\T_{\G,n})$ converges almost surely to $-\mathcal{J}_\G$, and consequently $g_n(t)$ converges almost surely to 0. 

Note that $\mathcal{J}_\G$ is positive definite by assumption {\bf A7}. Then, we can see that for all sufficiently large $n$ and any fixed $t$,
\begin{eqnarray}\label{ineq.Q2}
   \frac{1}{n}t'\paa\widetilde Q_n^{(\G)} (\bar\T_{\G,n})t \leq - \frac{1}{2}t'\mathcal{J}_\G t\quad a.s.
\end{eqnarray}
Moreover, for any $t\in \mathcal{R}_1$, we have $\|\T-\T_0\| \leq \|\hat\T_{\G,n}-\T_0\|+\|t\|/\sqrt{n} \leq \|\hat\T_{\G,n}-\T_0\|+ \delta_0$. 
Since $\pi(\T)$ is continuous at $\T_0$ and $\|\hat\T_{\G,n}-\T_0\|$ converges almost surely to zero by Theorem \ref{thm1}, for any $c>0$, we can choose $\delta_0>0$ small enough such that, for all sufficiently large $n$,
\begin{eqnarray}\label{ineq.pi}
\sup_{t\in\mathcal{R}_1} \Big| \pi\Big(\hat\T_{\G,n}+\frac{t}{\sqrt{n}}\Big) - \pi(\T_0)\Big|
= \sup_{\| \T-\hat\T_{\G,n}\| \le \delta_0} \Big| \pi(\T) - \pi(\T_0)\Big| \leq c.   
\end{eqnarray}
Thus, we have  from (\ref{Q2}), (\ref{ineq.Q2}), and  (\ref{ineq.pi}) that for all sufficiently large $n$ and any fixed $t\in\mathcal{R}_1$,
\begin{eqnarray*}
      \big|g_n(t)\big|
    &\leq&  \pi\Big( \hat\T_{\G,n} + \frac{t}{\sqrt{n}}\Big)\exp\Big( \frac{1}{2n} t'\paa \widetilde Q_n^{(\G)} (\bar\T_{\G,n})t \Big)
    +  \pi(\T_0)e^{-\frac{1}{2}t'\mathcal{J}_\G t} \\
    &\leq&
    \big(\pi(\T_0)+c \big)e^{-\frac{1}{4}t'\mathcal{J}_\G t} + \pi(\T_0)e^{-\frac{1}{2}t'\mathcal{J}_\G t} \quad a.s.,
\end{eqnarray*}
which is integrable on $\mathcal{R}_1$. Hence, by the dominate convergence theorem, it follows that
\begin{eqnarray}\label{conv.R1}
    \int_{\mathcal{R}_1} \big|g_n(t)\big|dt = o(1)\quad a.s.
\end{eqnarray}

Next, using Lemma \ref{Q.R2}, we can take a constant $\ep>0$ such that 
\[ \sup_{\|\T-\T_0\| >\delta_0} \frac{1}{n} \big(\widetilde Q_n^{(\G)}(\T) - \tilde Q_n^{(\G)}(\T_0)\big) < -\ep \quad a.s.\]
Then, for each $t\in\mathcal{R}_2$, we have by (\ref{conv.Q2}) and the consistency of $\hat\T_{\G,n}$ that for all sufficiently large $n$,
\begin{eqnarray*}
    \frac{1}{n} \Big(\widetilde Q_n^{(\G)}\Big( \hat\T_{\G,n} + \frac{t}{\sqrt{n}}\Big)- \widetilde Q_n^{(\G)}\big( \hat\T_{\G,n}\big)\Big)
&=& \frac{1}{n}\Big( \widetilde Q_n^{(\G)}(\T) - \widetilde Q_n^{(\G)}(\T_0)\Big)+ \frac{1}{n}\Big(\widetilde Q_n^{(\G)}(\T_0) - \widetilde Q_n^{(\G)}\big( \hat\T_{\G,n}\big)\Big)\\
&\leq& -\ep + \frac{1}{2n}(\T_0-\hat\T_{\G,n})'\paa \widetilde Q_n^{(\G)}\big(\bar\T_{\G,n} \big) (\T_0-\hat\T_{\G,n})\quad a.s.\\
&\leq& -\frac{\ep}{2}\quad a.s.,
\end{eqnarray*}
where $\bar\T_{\G,n}$ is between $\T_0$ and $\hat\T_{\G,n}$. Thus, we have that for all sufficiently large $n$,
\begin{eqnarray*}
    \int_{\mathcal{R}_2} \big|g_n(t)\big|dt
    &\leq& 
    e^{-\frac{n\ep}{2}}\int_{\mathcal{R}_2} \pi\Big( \hat\T_{\G,n} + \frac{t}{\sqrt{n}}\Big)dt + \int_{\mathcal{R}_2} \pi(\T_0)e^{-\frac{1}{2}t'\mathcal{J}_\G t}dt\quad a.s.\\
    &\leq&
    n^{d/2} e^{-\frac{n\ep}{2}} + \pi(\T_0)\int_{\mathcal{R}_2} e^{-\frac{1}{2}t'\mathcal{J}_\G t}dt =o(1),
\end{eqnarray*}
which together with (\ref{conv.R1}) yields (\ref{main}). This completes the proof.
\hfill{$\Box$}\vspace{0.2cm}\\

\noindent{\bf Proof of Theorem \ref{thm3}}\\
Following the same steps as in the proof of Theorem \ref{thm2}, and using the finiteness of $\int \|\T\| \pi(\T)d\T$, we can obtain that
\begin{equation*}
     \int_{\mathbb{R}^d} \|t\|\Big| \pi_\G(t\mid\bX) - \frac{|\mathcal{J}_\G|^{1/2}}{(2\pi)^{d/2}} e^{- \frac{1}{2} t' \mathcal{J}_\G t} \Big| dt =o(1)\quad a.s.,
\end{equation*}
which implies that 
\begin{equation*}
     \int_{\mathbb{R}^d} t \pi_\G(t\mid\bX)dt \stackrel{a.s.}{\longrightarrow}  \int_{\mathbb{R}^d} t \frac{|\mathcal{J}_\G|^{1/2}}{(2\pi)^{d/2}} e^{- \frac{1}{2} t' \mathcal{J}_\G t}dt=0.
\end{equation*}
Therefore, since $\hat\T_{\G,n}^{EDPE}=\E(\hat\T_{\G,n} + t/\sqrt{n} \mid \bX)$, we have 
\[\sqrt{n}(\hat\T_{\G,n}^{EDPE} - \hat\T_{\G,n})= \int_{\mathbb{R}^d} t \pi_\G(t \mid\bX)dt\ \longrightarrow\ 0\quad a.s.\]
The second result follows directly from Theorem \ref{thm1} and Slutsky’s theorem. This completes the proof.
\hfill{$\Box$}\vspace{0.2cm}\\

\begin{lemma}\label{lem:GARCH_C3}
For the objective function $\widetilde Q_n^{(\G)}(\T)$ given in \eqref{obj_GARCH}, there exists a positive random variable $V_n$ independent of $\theta$ such that, for all sufficiently small $\gamma>0$,
\[
\sup_{\theta \in \Theta}
\left|
\widetilde Q_n^{(\gamma)}(\theta)- K_n(\gamma)-
\sum_{t=1}^n \log \tilde f_\theta(X_t \mid \mathcal F_{t-1})
\right|
\le V_n \gamma,
\]
where $K_n(\gamma) = n(\gamma^{-1} - 1)$.
\end{lemma}
\begin{proof}
 Let $w_t(\theta) := (2\pi \tilde{\sigma}_t^2(\theta))^{-1/2}$ and $B_t(\theta) := X_t^2 / (2\tilde{\sigma}_t^2(\theta))$. The DPD-based objective function is given by
\begin{equation*}
    \widetilde{Q}_n^{(\gamma)}(\theta) = \sum_{t=1}^n w^\G_t(\theta) \left\{ \frac{1}{\gamma} \exp(-\gamma B_t(\theta)) - \left( \frac{1}{1+\gamma} \right)^{\frac 3 2} \right\}.
\end{equation*}
For a sufficiently small $\gamma > 0$, we can expand the terms inside the summation as follows:
\begin{align*}
    &w^\G_t(\theta) = 1 + \gamma \log w_t(\theta) + \frac{\gamma^2}{2}(\log w_t(\theta))^2 + O(\gamma^3),\\
    &\frac{1}{\gamma} \exp(-\gamma B_t(\theta)) = \frac{1}{\gamma} - B_t(\theta) + \frac{\gamma}{2}B^2_t(\theta) + O(\gamma^2),\\
    &(1+\gamma)^{-3/2} = 1 - \frac{3}{2}\gamma + O(\gamma^2).
\end{align*}  
Recall that the parameter space $\Theta$ is a compact subset of $(0,\infty)\times[0,\infty)^{p+q}$. Since $\tilde{\sigma}_t^2(\theta)$ is continuous in $\T$, we can see that $\inf_{\theta \in \Theta} \tilde{\sigma}_t^2(\theta) > 0$ and $\sup_{\theta \in \Theta} \tilde{\sigma}_t^2(\theta) < \infty$ for any fixed sample $\{X_t\}_{t=1}^n$. Consequently, $B_t(\theta)$ and $\log w_t(\theta)$ are uniformly bounded on $\Theta$, ensuring that the remainder terms in the expansions are also uniformly bounded in $\theta \in \Theta$.

Substituting these expansions into the summand $\tilde q_t^{(\gamma)}(\theta)$, we have
\begin{align*}
    \tilde q_t^{(\gamma)}(\theta) &= \left( 1 + \gamma \log w_t(\theta) + \frac{\gamma^2}{2}(\log w_t(\theta))^2 + O(\gamma^3) \right) \left\{ \left( \frac{1}{\gamma} - 1 \right) - B_t(\theta) + \gamma \left( \frac{B_t^2(\theta) + 3}{2} \right) + O(\gamma^2) \right\} \nonumber \\
    &= \left( \frac{1}{\gamma} - 1 \right) + (\log w_t(\theta) - B_t(\theta)) + \gamma R_t(\theta) + O(\gamma^2),
\end{align*}
where $R_t(\theta) := \frac{1}{2}(B_t^2(\theta) + 3) - (1+B_t(\theta))\log w_t(\theta) + \frac{1}{2}(\log w_t(\theta))^2$. Summing over $t=1, \dots, n$ yields
\begin{equation*}
    \widetilde{Q}_n^{(\gamma)}(\theta) = K_n(\gamma) + \widetilde{Q}_n^{(0)}(\theta) + \gamma \sum_{t=1}^n R_t(\theta) + O(\gamma^2),
\end{equation*}
Since $R_t(\theta)$ is continuous on the compact set $\Theta$, $\sup_{\theta \in \Theta} |\sum_{t=1}^n R_t(\theta)| < \infty$. Thus, we can take a positive random variable $V_n$, depending only on the sample $\{X_t\}_{t=1}^n$, such that for sufficiently small $\gamma$,
\begin{equation*}
    \sup_{\theta \in \Theta} \left| \tilde{Q}_n^{(\gamma)}(\theta) - K_n(\gamma) - \tilde{Q}_n^{(0)}(\theta) \right| \le \gamma \left( \sup_{\theta \in \Theta} \left| \sum_{t=1}^n R_t(\theta) \right| + O(\gamma) \right) \le V_n \gamma,
\end{equation*}
which completes the proof.
 \end{proof}

\bibliography{reference.bib}

@article{ling2010general,
  title={A general asymptotic theory for time-series models},
  author={Ling, Shiqing and McAleer, Michael},
  journal={Statistica Neerlandica},
  volume={64},
  number={1},
  pages={97--111},
  year={2010},
  publisher={Wiley Online Library}
}

@article{song2021test,
  title={Test for parameter change in the presence of outliers: the density power divergence-based approach},
  author={Song, Junmo and Kang, Jiwon},
  journal={Journal of Statistical Computation and Simulation},
  volume={91},
  number={5},
  pages={1016--1039},
  year={2021},
  publisher={Taylor \& Francis}
}

@article{warwick2005choosing,
  title={Choosing a robustness tuning parameter},
  author={Warwick, Jane and Jones, MC},
  journal={Journal of Statistical Computation and Simulation},
  volume={75},
  number={7},
  pages={581--588},
  year={2005},
  publisher={Taylor \& Francis}
}

@article{ruli2020robust,
  author  = {Ruli, Edoardo and Sartori, Nicola and Ventura, Laura},
  title   = {Robust Approximate Bayesian Inference},
  journal = {Journal of Statistical Planning and Inference},
  volume  = {202},
  pages   = {1--16},
  year    = {2020},
  doi     = {10.1016/j.jspi.2019.06.002}
}

@article{ohagan2012conflict,
  author  = {O'Hagan, Anthony and Pericchi, Luis},
  title   = {Bayesian Statistics and the Problem of Prior--Data Conflict},
  journal = {Statistical Science},
  volume  = {27},
  number  = {1},
  pages   = {1--31},
  year    = {2012},
  doi     = {10.1214/11-STS362}
}

@article{andrade2011bayesian,
  author  = {Andrade, Jos{\'e} Ailton Alencar and O'Hagan, Anthony},
  title   = {Bayesian Robustness Modelling of Location and Scale Parameters},
  journal = {Scandinavian Journal of Statistics},
  volume  = {38},
  number  = {4},
  pages   = {691--711},
  year    = {2011},
  doi     = {10.1111/j.1467-9469.2011.00750.x}
}

@article{song2017robusta,
  author  = {Song, Junmo},
  title   = {Robust Estimation of Dispersion Parameter in Discretely Observed Diffusion Processes},
  journal = {Statistica Sinica},
  volume  = {27},
  number  = {1},
  pages   = {373--388},
  year    = {2017},
  doi     = {10.5705/ss.2014.162}
}

@article{ghosh2013robust,
  author  = {Ghosh, Abhik and Basu, Ayanendranath},
  title   = {Robust Estimation for Independent Non-Homogeneous Observations Using Density Power Divergence with Applications to Linear Regression},
  journal = {Electronic Journal of Statistics},
  volume  = {7},
  pages   = {2420--2456},
  year    = {2013},
  doi     = {10.1214/13-EJS847}
}

@article{loh2024robust,
  author  = {Loh, Po-Ling},
  title   = {A Theoretical Review of Modern Robust Statistics},
  journal = {Annual Review of Statistics and Its Application},
  volume  = {11},
  pages   = {1--28},
  year    = {2024},
  doi     = {10.1146/annurev-statistics-040720-012424}
}

@book{huber1981robust,
  author    = {Huber, Peter J.},
  title     = {Robust Statistics},
  publisher = {Wiley},
  address   = {New York},
  year      = {1981}
}

@article{ronchetti2021robust,
  author  = {Ronchetti, Elvezio},
  title   = {The Main Contributions of Robust Statistics to Statistical Science and a New Challenge},
  journal = {METRON},
  volume  = {79},
  number  = {2},
  pages   = {127--135},
  year    = {2021},
  doi     = {10.1007/s40300-021-00204-6}
}

@article{straumann2006quasi,
  title={Quasi-maximum-likelihood estimation in conditionally heteroscedastic time series: A stochastic recurrence equations approach},
  author={Straumann, Daniel and Mikosch, Thomas},
  year={2006}
}

@article{hamadeh2011asymptotic,
  title={Asymptotic properties of LS and QML estimators for a class of nonlinear GARCH processes},
  author={Hamadeh, Tawfik and Zako{\"\i}an, Jean-Michel},
  journal={Journal of Statistical Planning and Inference},
  volume={141},
  number={1},
  pages={488--507},
  year={2011},
  publisher={Elsevier}
}

@article{francq:zakoian:2004,
	author = {Francq, C. and Zako\"{i}an, J.},
	journal = {Bernoulli},
	pages = {605-637},
	title = {Maximum likelihood estimation of pure {GARCH} and {ARMA-GARCH} processes},
	volume = {10},
	year = {2004}}

@article{basu1998,
  title={{Robust and efficient estimation by minimising a density power divergence}},
  author={Basu, Ayanendranath and Harris, Ian R and Hjort, Nils L and Jones, MC},
  journal={Biometrika},
  volume={85},
  number={3},
  pages={549--559},
  year={1998},
  publisher={Oxford University Press}
}

@article{ghosh2016robust,
  title={{Robust Bayes estimation using the density power divergence}},
  author={Ghosh, Abhik and Basu, Ayanendranath},
  journal={Annals of the Institute of Statistical Mathematics},
  volume={68},
  number={2},
  pages={413--437},
  year={2016},
  publisher={Springer}
}

@article{hooker2014bayesian,
  title={Bayesian model robustness via disparities},
  author={Hooker, Giles and Vidyashankar, Anand N},
  journal={Test},
  volume={23},
  number={3},
  pages={556--584},
  year={2014},
  publisher={Springer}
}

@article{fujisawa2008robust,
  title={Robust parameter estimation with a small bias against heavy contamination},
  author={Fujisawa, Hironori and Eguchi, Shinto},
  journal={Journal of Multivariate Analysis},
  volume={99},
  number={9},
  pages={2053--2081},
  year={2008},
  publisher={Elsevier}
}

@article{nakagawa2020robust,
  title={Robust Bayesian inference via $\gamma$-divergence},
  author={Nakagawa, Tomoyuki and Hashimoto, Shintaro},
  journal={Communications in Statistics-Theory and Methods},
  volume={49},
  number={2},
  pages={343--360},
  year={2020},
  publisher={Taylor \& Francis}
}

@article{lee2009minimum,
  title={Minimum density power divergence estimator for {GARCH} models},
  author={Lee, Sangyeol and Song, Junmo},
  journal={Test},
  volume={18},
  pages={316--341},
  year={2009},
  publisher={Springer}
}

@article{song2021sequential,
  title={Sequential change point test in the presence of outliers: the density power divergence based approach},
  author={Song, Junmo},
  journal={Electronic Journal of Statistics},
  volume={15},
  number={1},
  pages={3504--3550},
  year={2021},
  publisher={The Institute of Mathematical Statistics and the Bernoulli Society}
}

@article{hoffman2014no,
  title={The No-U-Turn sampler: adaptively setting path lengths in Hamiltonian Monte Carlo.},
  author={Hoffman, Matthew D and Gelman, Andrew and others},
  journal={J. Mach. Learn. Res.},
  volume={15},
  number={1},
  pages={1593--1623},
  year={2014}
}

@article{bai2014efficient,
  title={Efficient pairwise composite likelihood estimation for spatial-clustered data},
  author={Bai, Yun and Kang, Jian and Song, Peter X-K},
  journal={Biometrics},
  volume={70},
  number={3},
  pages={661--670},
  year={2014},
  publisher={Oxford University Press}
}

\clearpage

\end{document}